\documentclass[12pt]{amsart}

\usepackage{amsmath,amsfonts,amssymb,mathrsfs,stmaryrd} 
\usepackage[margin=3cm]{geometry}
\usepackage{color}
\usepackage{graphics}
\usepackage{enumerate}
\usepackage{amscd} 
\usepackage{amsxtra}
\usepackage{epic,eepic}
\usepackage{subsupscripts} 
\usepackage[hypertex]{hyperref} 

\newtheorem{theorem}{Theorem}[section]
\newtheorem{corollary}[theorem]{Corollary}
\newtheorem{lemma}[theorem]{Lemma}

\newtheorem{proposition}[theorem]{Proposition}
\newtheorem{assumption}[theorem]{Assumption}
\newtheorem{conjecture}[theorem]{Conjecture} 

\theoremstyle{definition} 
\newtheorem{definition}[theorem]{Definition}
\newtheorem{example}[theorem]{Example}

\newtheorem{remark}[theorem]{Remark}
\newtheorem{notation}[theorem]{Notation} 

\def\parfrac#1#2{\frac{\partial{#1}}{\partial #2}}

\newcommand{\cA}{\mathcal{A}}
\newcommand{\cE}{\mathcal{E}}
\newcommand{\ocE}{\overline{\cE}}
\newcommand{\cK}{\mathcal{K}}
\newcommand{\cL}{\mathcal{L}}
\newcommand{\cM}{\mathcal{M}}
\newcommand{\hcM}{\widehat{\cM}} 
\newcommand{\cN}{\mathcal{N}}
\newcommand{\cO}{\mathcal{O}}
\newcommand{\cP}{\mathcal{P}}
\newcommand{\cR}{\mathcal{R}}
\newcommand{\hcR}{\widehat{\cR}}
\newcommand{\cU}{\mathcal{U}}
\newcommand{\cS}{\mathcal{S}} 
\newcommand{\hcS}{\widehat{\cS}}
\newcommand{\hD}{\widehat{D}}
\newcommand{\hV}{\widehat{V}} 
\newcommand{\hL}{\widehat{L}}
\newcommand{\hbeta}{\hat{\beta}} 
\newcommand{\halpha}{\hat{\alpha}} 
\newcommand{\hSigma}{\widehat{\Sigma}} 
\newcommand{\hS}{\widehat{S}} 
\newcommand{\hu}{\hat{u}} 

\newcommand{\hzero}{\hat{0}} 
 
\newcommand{\hE}{\widehat{E}}
\newcommand{\hs}{\hat{s}}
\newcommand{\hpsi}{\hat{\psi}}

\newcommand{\tD}{\widetilde{D}} 

\newcommand{\tS}{\widetilde{S}} 
\newcommand{\tp}{\tilde{p}}

\newcommand{\tSigma}{\widetilde{\Sigma}} 
\newcommand{\tK}{\widetilde{K}}
\newcommand{\tu}{\tilde{u}}

\newcommand{\ol}[1]{\overline{#1}}

\newcommand{\ran}{\right\rangle}
\newcommand{\lan}{\left\langle}

\newcommand{\MM}{\mathfrak{M}}
\newcommand{\mm}{\mathfrak{m}} 
\newcommand{\frg}{\mathfrak{g}} 
\newcommand{\frt}{\mathfrak{t}} 
\newcommand{\frf}{\mathfrak{f}} 
\newcommand{\fra}{\mathfrak{a}} 
\newcommand{\frb}{\mathfrak{b}} 
\newcommand{\frr}{\mathfrak{r}}
\newcommand{\frs}{\mathfrak{s}} 
\newcommand{\frS}{\mathfrak{S}} 
\newcommand{\frR}{\mathfrak{R}} 
\newcommand{\frK}{\mathfrak{K}} 
\newcommand{\frU}{\mathfrak{U}} 

\newcommand{\Spec}{\operatorname{Spec}}
\newcommand{\Pic}{\on{Pic}}
 
\newcommand{\Jac}{\on{Jac}}
\newcommand{\Int}{\on{Int}}

\renewcommand{\P}{\bb{P}} 
\newcommand{\C}{\bb{C}}
\newcommand{\D}{\bb{D}}
\newcommand{\bL}{\bb{L}}
\newcommand{\Z}{\bb{Z}}
\newcommand{\R}{\bb{R}}
\newcommand{\Q}{\bb{Q}}
\newcommand{\T}{\bb{T}}
\newcommand{\G}{\bb{G}}

\newcommand{\F}{\bb{F}} 
\newcommand{\E}{\bb{E}}
\newcommand{\bE}{\mathbf{E}} 

\newcommand{\quu}{/\kern-.7ex/\kern-.7ex/}

\newcommand{\on}{\operatorname} 
\newcommand{\Aut}{ \on{Aut} }
\newcommand{\Hom}{ \on{Hom} }
\newcommand{\Ker}{ \on{Ker} }
\newcommand{\ev}{\on{ev}} 
\newcommand{\vdim}{\on{vir. dim}}

\newcommand{\NE}{\on{NE}}
\newcommand{\Bl}{\on{Bl}} 
\newcommand{\Ham}{\on{Ham}} 
\newcommand{\pr}{\on{pr}}
\newcommand{\id}{\on{id}} 
 
\renewcommand{\Im}{\on{Im}}
\newcommand{\Map}{\on{Map}} 

\newcommand{\KS}{\operatorname{KS}} 
\newcommand{\ks}{\operatorname{ks}} 
\newcommand{\frks}{\mathfrak{ks}}

\newcommand{\bb}[1]{\ensuremath{\mathbb{#1}}}
\newcommand{\PD}{{\rm PD}}

\title[Seidel elements and potential functions]
{Seidel elements and potential functions of holomorphic disc counting} 

\date{\today}

\author{Eduardo Gonz\'alez}
\address{Department of Mathematics\\
UMASS Boston\\100 Morrisey Blvd\\
Boston, MA 02125. USA.}
\email{eduardo@math.umb.edu}

\author{Hiroshi Iritani}
\address{Department of Mathematics, Graduate School of Science\\
  Kyoto University\\
  Oiwake-cho, Kitashirakawa, Sakyo-ku, Kyoto,
  606-8502, Japan.}
\email{iritani@math.kyoto-u.ac.jp}

\begin{document}
\maketitle
\begin{abstract}
  Let $M$ be a symplectic manifold equipped with a
  Hamiltonian circle action and let $L$ be an
  invariant Lagrangian submanifold of $M$.  We
  study the problem of counting holomorphic
  \emph{disc sections} of the trivial $M$-bundle
  over a disc with boundary in $L$ through
  degeneration.  We obtain a conjectural
  relationship between the potential function of
  $L$ and the Seidel element associated to the
  circle action.  When applied to a Lagrangian
  torus fibre of a semi-positive toric manifold,
  this degeneration argument reproduces a
  conjecture (now a theorem) of
  Chan-Lau-Leung-Tseng \cite{CLLT11,CLLT12}
  relating certain correction terms appearing in
  the Seidel elements with the potential function.
\end{abstract}



\section{Introduction}

Let $M$ be a symplectic manifold with a
Hamiltonian circle action.  Seidel \cite{Se-pi97}
constructed an invertible element of the quantum
cohomology of $M$ by counting pseudo-holomorphic
sections of the associated $M$-bundle $E$ over
$S^2$:
\[
E = (M \times S^3) / S^1 
\]
where $S^1$ acts by the diagonal action and $S^3
\to S^2$ is the Hopf fibration.  Seidel elements
have been used to detect
essential 
loops in the group $\Ham(M,\omega)$ of Hamiltonian
diffeomorphisms.  McDuff-Tolman \cite{McTo-To06}
used them to verify Batyrev's presentation of
quantum cohomology rings for toric varieties.

In a previous paper \cite{GoIr-Se12}, we
computed Seidel elements of semi-positive toric
manifolds and found that they are closely related
to Givental's mirror transformation
\cite{Gi-A-98}.  Chan-Lau-Leung-Tseng
\cite{CLLT11} conjectured that certain correction
terms appearing in our computation of Seidel
elements determine the potential function of a
Lagrantian torus fibre.  The potential function
here is given by counting holomorphic discs with
boundary in a Lagrangian torus fibre and is
thought of as a mirror of the toric variety.  The
conjecture was proved by themselves \cite{CLLT12}
in a recent preprint.  In this paper, we propose
an alternative approach which relates Seidel
elements and potential functions via
\emph{degeneration}.  Our method should apply to a
general symplectic manifold $M$ with a Hamiltonian
$S^1$-action and an invariant Lagrangian.

We assume that $M$ is a smooth projective variety,
equipped with a $\C^\times$-action and an
$S^1$-invariant K\"{a}hler form $\omega$.  Let $L$
be an $S^1$-invariant Lagrangian submanifold of
$M$.  Let $\cM_1(\beta)$ denote the moduli space
of genus-zero bordered stable holomorphic maps
from $(\Sigma, \partial\Sigma)$ to $(M,L)$ with
one boundary marking and representing $\beta \in
H_2(M,L)$.  By the fundamental work of
Fukaya-Oh-Ohta-Ono \cite{FuOh-La09,
  FOOO-technical12}, $\cM_1(\beta)$ is compact and
carries a Kuranishi structure with boundary and
corner.  Let $\beta$ be a class of Maslov index
two.  Under certain assumptions (see \S
\ref{subsec:potential-general}), the virtual
fundamental \emph{chain} of $\cM_1(\beta)$ is a
\emph{cycle} of dimension $\dim_\R L$ and one can
define the \emph{open Gromov-Witten invariant}
$n_\beta\in \Q$ by
\[
\ev_*[\cM_1(\beta)]^{\rm vir} = n_\beta [L] 
\]
where $\ev\colon\cM_1(\beta) \to L$ is the evaluation map. 
The potential function $W$ is 
\[
W = \sum_{\beta \in H_2(M,L) : \mu(\beta) =2}
n_\beta z^\beta.
\]
The idea of degeneration is that instead of
counting discs in $(M,L)$, we consider the problem
of counting \emph{disc sections} of the trivial
bundle $M\times \D \to \D$ with boundary in
$L\times S^1$.  Then we degenerate the target
$M\times \D$ to the union $E\cup_M (M\times \D)$.
From this geometry we expect the following
\emph{degeneration formula} (see \S
\ref{subsec:degeneration} for details):
\begin{equation} 
\label{eq:degformula-introd} 
\varphi_*\ev_*[\cM_1(\hbeta)]^{\rm vir} 
= \sum_{r(\hbeta) = \sigma + \halpha} 
\ev_*[\cM_{\rm S}(\sigma) \times_M \cM^{\rm
  rel}_{1,1}(\halpha)]^{\rm vir}  
\end{equation} 
\emph{if} both-hand sides carry virtual fundamental 
\emph{cycles}, instead of \emph{chains}.  
Here $\hbeta \in H_2(M\times \D, L\times S^1)$ denotes 
a disc section class corresponding to $\beta \in H_2(M,L)$ 
and $\cM_1(\hbeta)$ is the corresponding moduli space of 
disc sections. 
The summation in the right-hand side is taken over 
all possible decompositions $\sigma + \halpha$ 
of the class $\hbeta$ into a section class $\sigma$ of $E$ 
and another disc section class $\halpha$ under the degeneration. 
Also $\cM_{\rm S}(\sigma)$ is a moduli space of 
holomorphic sections of $E$ in the class $\sigma$, 
which is relevant to the Seidel element. 
This formula relates disc counts of different boundary types; 
the boundary classes $\partial \alpha$ and $\partial \beta$ 
from the both-hand sides differ exactly by the $S^1$-action.

The degeneration formula predicts a relationship 
between the Seidel element of the $S^1$-action and 
the potential function $W$. 
We need the following conditions in order to extract 
meaningful information from the formula \eqref{eq:degformula-introd}: 
\begin{itemize} 
\item[(i)] $\cM_1(\beta)$ is empty for all $\beta
\in H_2(M,L) $ with $\mu(\beta) \le 0$.  

\item[(ii)] The maximal fixed component $F_{\rm max}\subset M$ of 
the $\C^\times$-action (see \S \ref{subsec:Seidel}) 
is of complex codimension one and the $\C^\times$-weight 
on the normal bundle is $-1$. 

\item[(iii)] $c_1(M)$ is semi-positive. 

\item[(iv)] $\ev(\cM_{\rm S}(\sigma))$ is disjoint from $L$ 
for all $\sigma\in H_2^{\rm sec}(E)$ such that 
$\lan c_1^{\rm vert}(E), \sigma \ran = -1$.  
\end{itemize} 
\begin{theorem}[Corollary \ref{cor:degeneration}] 
Assume that $M$ is simply-connected and $L$ is connected. 
Assume that the degeneration formula \eqref{eq:degformula-introd} 
holds (see Conjecture \ref{conj:degeneration} for a precise formulation) 
and that the above conditions (i)--(iv) are satisfied. 
Then 
\begin{equation*} 
z^{\alpha_0} = \langle \hS^{(2)}, dW \rangle + \tS^{(0)}  
\end{equation*} 
holds in a certain ``open" Novikov ring $\Lambda^{\rm op}$ 
(see \S \ref{subsec:potential-general}), where 
\begin{itemize} 
\item $\alpha_0 \in H_2(M,L)$ is the maximal disc class 
defined by rotating a path connecting $L$ 
and $F_{\rm max}$ by the $S^1$-action 
(see \S \ref{subsec:relativehomology}); 

\item $dW = \sum_{\mu(\beta)=2} \beta \otimes n_\beta z^\beta$ 
is the logarithmic derivative of $W$; 

\item $\tS = \tS^{(0)} +  \tS^{(2)}$ is the Seidel element 
associated to the $S^1$-action and $\tS^{(i)} 
\in H^i(M) \otimes \Lambda$ 
($\Lambda$ is the ``closed" Novikov ring in Remark 
\ref{rem:Seidelrep}); 

\item $\hS^{(2)} \in H^2(M,L)\otimes \Lambda$ is 
a lift of $\tS^{(2)}$ (see Definition \ref{def:Seidellift}). 
\end{itemize} 
In particular, 
\[
\KS(\tS) = [z^{\alpha_0}] 
\]
holds in a certain Jacobi algebra of $W$, where $\KS$ denotes 
the Kodaira-Spencer mapping (see the end of \S \ref{subsubsec:conjecture}). 
\end{theorem} 


In the second half of the paper, we apply these to
a semi-positive toric manifold $X$ and calculate
the potential function of a Lagrangian torus fibre
$L \subset X$.  In toric case, the potential
function can be regarded as a function on the
moduli space $\MM_{\text{opcl}}$ of Lagrangian
torus fibres $L$ together with complexfied
K\"{a}hler classes $-\omega+iB$ and lifts $h\in
H^2(X,L;U(1))$ of $\exp(iB)$ (see \S
\ref{subsubsec:open-closed}, $h$ defines a
$U(1)$-local system on $L$ when $B=0$).  The
potential function is of the form:
\[
W= w_1 + \cdots + w_m 
\]
with $w_i = f_i(q) z_i$, where $f_i(q)\in\Lambda$ 
is the \emph{correction term} defined by   
\[
f_i(q) = \sum_{d \in H_2(X;\Z): \lan c_1(X), d \ran =0} 
n_{\beta_i + d} q^d. 
\]
Each term $w_i$ corresponds to a prime toric divisor $D_i \subset X$ 
and arises from disc counting of fixed boundary type $b_i \in H_1(L)$. 
Applying the degeneration formula, we get: 
\begin{theorem}[Theorem \ref{thm:degeneration-toric}] 
\label{thm:degformula-toric-introd} 
Assume that the degeneration formula \eqref{eq:degformula-introd} 
(Conjecture \ref{conj:degeneration}) holds 
for $(X,L)$ equipped with the $\C^\times$-action $\rho_j$ 
rotating around the prime toric divisor $D_j$ (see \S \ref{subsec:Seidel-toric}). 
Let $\tS_j\in H^2(X)\otimes \Lambda$ be 
the Seidel element $\rho_j$ and let 
$\hS_j\in H^2(X,L)\otimes \Lambda$ be its lift.  Then we have 
\begin{equation*} 
  \langle  \hS_j, dw_k\rangle =\delta_{jk} z_j. 
\end{equation*} 
In particular we have $\KS(\tS_j) =[z_j]$.  
\end{theorem}

We observe in Theorem \ref{thm:CLLT} 
that the degeneration formula reproduces the following 
conjecture (now a theorem) of Chan-Lau-Leung-Tseng \cite{CLLT11,CLLT12}. 
\begin{theorem}[{\cite[Conjecture 4.12]{CLLT11}, 
\cite[Theorem 1.1]{CLLT12}}] 
\label{thm:CLLT-introd} 
Let $g_0^{(j)}(y)$, $j=1,\dots,m$ 
be explicit hypergeometric functions  
in variables $y_1,\dots,y_r$ ($r= \dim H^2(X)$) 
given in equation \eqref{eq:g0j}. 
Then we have 
\[
f_j(q) = \exp \left(g_0^{(j)}(y)\right) 
\]
under an explicit change of variables (mirror transformation) 
of the form $\log q_i = \log y_i + g_i(y)$, $i=1,\dots,r$ 
with $g_i(y) \in \Q[\![y_1,\dots,y_r]\!]$ and 
$g_i(0)=0$. 
\end{theorem} 

In \cite{GoIr-Se12}, we introduced \emph{Batyrev elements} 
$\tD_j$ as mirror analogues of the divisor classes $D_j$. 
They satisfy the relations of Batyrev's quantum ring \cite{Ba-Qu93} 
for toric varieties. 
The hypergeometric functions $g_0^{(j)}(y)$ originally 
appeared in our computation \cite{GoIr-Se12} 
as the difference between the Seidel and the Batyrev elements: 
\[
\tD_j = \exp\left(g_0^{(j)}(y)\right) \tS_j.  
\] 
Hence by Theorem \ref{thm:degformula-toric-introd},  
$\tS_j$ and $\tD_j$ correspond respectively to 
$[z_j]$ and $[w_j]$ under the Kodaira-Spencer mapping 
(see also \cite[Theorem 1.5]{CLLT12}). 

Finally we discuss briefly the method of 
Chan-Lau-Leung-Tseng \cite{CLLT12}. 
Their approach is different from ours 
but is closely related. 
They observed that a holomorphic disc in $(X,L)$ 
whose boundary class is $b_j\in H_1(L)$ can be completed 
to a holomorphic \emph{sphere} in the $M$-bundle 
$E'_j$ associated to the inverse $\C^\times$-action 
$\rho_j^{-1}$. 
Using this, they identified open Gromov-Witten invariants of $(X,L)$ 
with certain closed invariants of $E'_j$. 
The associated bundle $E'$ of the inverse action 
also appears in our story as the central fibre $E\cup_M E'$ 
of the degeneration of the closed manifold $M\times \P^1$ 
(instead of $M\times \D$) in \S \ref{ss:def}. 


\vspace{5pt} 
\noindent 
\textbf{Acknowledgments.}  We are grateful to
Kwokwai Chan, Siu-Cheong Lau, Naichung Conan Leung
and Hsian-Hua Tseng for their beautiful conjecture
\cite{CLLT11}, stimulating discussions and their
interests in our work.  We thank Kenji Fukaya and
Kaoru Ono for many helpful comments.  E.G. wants
to thank the mathematics departments at Kyoto
University and MIT for their hospitality while working and
revising this project. E.G. was supported by NSF grants
DMS-1104670 and DMS-1510518. H.I. is supported by 
JSPS KAKENHI Grant Number 22740042, 25400069, 
23224002.

\section{Preliminaries} 
\label{s:prelim}

In this section, we review a potential function of a 
Lagrangian submanifold and a Seidel element associated 
to a Hamiltonian circle action. 

\subsection{Potential function of a Lagrangian submanifold} 
\label{subsec:potential-general} 
The potential of a Lagrangian submanifold 
arises as the 0-th operation $\mm_0$ of the corresponding 
$A_\infty$-algebra in Lagrangian Floer theory of 
Fukaya-Oh-Ohta-Ono \cite{FuOh-La09}. 
In this paper, we do not use the full generality of 
$A_\infty$-formalism developed in \cite{FuOh-La09}; 
instead we consider potential functions 
under certain restrictive assumptions. 

Let $(M,\omega)$ be a closed symplectic manifold 
and $L$ be a Lagrangian submanifold. 
For simplicity, we restrict ourselves to the case where $M$ 
is a smooth projective variety. 
We assume that $L$ is oriented, relatively-spin 
and fix a relative spin structure \cite[Definition 8.1.2]{FuOh-La09} 
of $L$ so that the moduli space of bordered stable maps to 
$(M,L)$ has an oriented Kuranishi structure. 
Let $\mu \colon H_2(M,L) \to \Z$ denote the Maslov index. 
It takes values in $2\Z$ since $L$ is oriented. 

Let $\cM_1(\beta)$ denote the moduli space 
of stable holomorphic maps from a genus-zero 
bordered Riemann surface $(\Sigma, \partial \Sigma)$ to 
$(M,L)$ with one boundary marked point and 
in the class $\beta\in H_2(M,L)$. 
This was denoted by $\cM_1^{\rm main}(\beta)$ in \cite{FuOh-La09}.  
By \cite[Proposition 7.1.1]{FuOh-La09} (see also 
\cite[Theorem 15.3]{FOOO-technical12}), 
$\cM_1(\beta)$ is compact and equipped with an oriented 
Kuranishi structure (with boundary and corner) 
and has virtual dimension $n+ \mu(\beta) -2$, where $n= \dim_\R L$. 
Let $\ev \colon \cM_1(\beta) \to L$ denote the evaluation map. 
Define an open version of Novikov ring $\Lambda^{\rm op}$ to be 
the space of all formal power series 
\[
\sum_{\beta\in H_2(M,L)} c_\beta z^\beta
\]
with $c_\beta \in \Q$ such that 
\[
\sharp\left\{
\beta \,:\, c_\beta \neq 0, \ 
\textstyle \int_\beta \omega < E \right\} < \infty 
\]
holds for all $E\in \R$. 

\begin{definition} 
\label{def:W-general}
Assume that $\cM_1(\beta)$ is empty 
for all $\beta\in H_2(M,L)$ with $\mu(\beta)\le 0$. 
Then $\cM_1(\beta)$ with $\mu(\beta)=2$ has no boundary 
and carries a virtual fundamental \emph{cycle} 
of dimension $n=\dim_\R L$ \cite[Lemma A.1.32]{FuOh-La09}.
We define \emph{open Gromov-Witten invariants}  
$n_\beta\in \Q$ by 
\[
\ev_*[\cM_1(\beta)]^{\rm vir} = n_\beta [L] 
\]
for $\beta$ with $\mu(\beta)=2$, where $[L]\in H_n(L)$ is the 
fundamental class of $L$. 
The \emph{potential function} of $L$ is defined to be the 
formal sum: 
\[
W = \sum_{\beta \in H_2(M,L) : \mu(\beta)=2} 
n_\beta z^\beta.  
\]
This is an element of $\Lambda^{\rm op}$. 
\end{definition} 

We can decompose $W$ according to 
boundary classes of discs. 
\begin{definition} 
\label{def:W-gamma}
Under the same assumption as in Definition \ref{def:W-general}, 
we write 
\[
W = \sum_{\gamma \in H_1(L)} W_\gamma 
\]
with $W_\gamma \in \Lambda^{\rm op}$ given by 
\[
W_\gamma := \sum_{\beta \in H_2(M,L):\, 
\mu(\beta)=2, \partial \beta =\gamma} n_\beta z^\beta. 
\]
\end{definition} 

\begin{remark} 
The potential function does depend on the choice of a
complex structure on $M$ and this is a reason why we
restricted to a smooth projective variety $M$. For example,
the Hirzebruch surfaces $\F_0=\P^1\times\P^1$ and $\F_2$
together with their Lagrangian torus fibres are
symplectomorphic to each other, but the potential functions
are different. See Auroux \cite{Au-Mi07} for wall-crossing
of disc counting.  
\end{remark} 

\subsection{Seidel elements} 
\label{subsec:Seidel} 
Seidel element is an invertible element of quantum cohomology 
associated to a loop in the group $\Ham(M,\omega)$ 
of Hamiltonian diffeomorphisms of a symplectic manifold $(M,\omega)$. 
In this paper we restrict to the case where $M$ is a smooth projective variety 
equipped with an algebraic $\C^\times$-action. 
In this case, the associated $S^1$-action is Hamiltonian 
and yields a loop in $\Ham(M,\omega)$. 
We refer the reader to \cite{Se-pi97,LaMc-To99,Mc-Qu00} 
for the original definitions and to \cite{McTo-To06, Go-Qu06} 
for applications in symplectic topology. 

Let $M$ be a smooth projective variety, equipped with a 
$\C^\times$-action. 

\begin{definition}
\label{def:a-bundle} 
The \emph{associated bundle} of the
$\C^\times$-action on $M$ 
is the $M$-bundle over $\P^1$
\begin{equation*}
  E:=M \times (\C^{2}\setminus \{0\}) /
  \C^{\times}\to \P^1, 
\end{equation*}
where $\C^\times$ acts with the diagonal
action $\lambda\cdot (x, (z_1,z_2))=(\lambda x, (\lambda z_1, \lambda z_2))$.
\end{definition}

\begin{remark}\label{rk:ssidel}
In symplectic geometric terms, the associated
bundle is in fact a clutched bundle obtained by
gluing two trivial $M$-bundles over the unit disc, 
along the boundary, using the action. 
More precisely, 
\[ 
E=(M\times \D_0)\cup_g (M\times \D_\infty) 
\]
where $\D_0 = \{z\in \C \,:\, |z| \le 1\}$ and 
$\D_\infty = \{z\in \C\, :\, |z| \ge 1\} \cup \{\infty\}$ 
and the gluing map 
$g \colon M\times \partial \D_0 \to M\times \partial \D_\infty$ 
is given by 
\[
g(x,e^{i\theta}) = (e^{-i\theta} \cdot x, e^{i\theta}). 
\]
This construction can be generalized to 
a loop in the group of Hamiltonian diffeomorphims 
and yields a Hamiltonian bundle $E \to \P^1$ 
in general. 
One can equip a symplectic form $\omega_E$ 
on the total space $E$ of the Hamiltonian bundle 
such that $\omega_E$ restricts to the symplectic form 
$\omega_M$ on each fibre \cite{Se-pi97}. 
\end{remark}

By Atiyah's theorem \cite{At-Co82}, 
there exists a unique $\C^\times$-fixed component 
$F_{\rm max} \subset M^{\C^\times}$ 
such that the normal bundle of $F_{\rm max}$ 
has only negative $\C^\times$-weights.  
For a Hamiltonian function $H$ generating the $S^1$-action, 
$F_{\rm max}$ is the locus where $H$ takes the maximum value. 
Each fixed point $x\in M^{\C^\times}$ defines a section
$\sigma_x$ of $E$.  We denote by $\sigma_0$ the
section associated to a fixed point in $F_{\rm max}$.  
We call it a \emph{maximal section}. 
This defines a splitting\footnote
{The section $\sigma_0$ gives a splitting of the Serre 
spectral sequence. In general one has a non-canonical splitting 
$H^*(E;\Q) \cong H^*(M;\Q) \otimes H^*(\P^1;\Q)$ 
for any Hamiltonian bundle $E\to \P^1$ \cite{Mc-Qu00}.}
\begin{equation}
\label{eq:split_maximal}
H_2(E;\Z) \cong 
\Z [\sigma_0] \oplus H_2(M;\Z). 
\end{equation} 
Let $\NE(M)\subset H_2(M,\R)$ denote the Mori
cone, that is the cone generated by effective
curves and set $\NE(M)_\Z := \{d \in H_2(M;\Z) \,:\, 
d \in \NE(M)\}$.  
We introduce $\NE(E)$ and $\NE(E)_\Z$ similarly.
\begin{lemma}[{\cite[Lemma 2.2]{GoIr-Se12}}] 
$\NE(E)_\Z = \Z_{\ge 0} [\sigma_0] + \NE(M)_\Z$. 
\end{lemma} 

Let $H_2^{\sec}(E;\Z)$ denote the affine 
subspace of $H_2(E;\Z)$ 
which consists of section classes, i.e.\ 
the classes that project to the positive generator 
of $H_2(\P^1;\Z)$. 
We set $\NE(E)^{\sec}_\Z := 
\NE(E)_\Z \cap H_2^{\sec}(E;\Z)$. 
The above lemma shows that 
\begin{equation}
\label{eq:split-sigma0}
\NE(E)^{\sec}_\Z=  [\sigma_0] + \NE(M)_\Z.  
\end{equation} 
For $d\in \NE(X)_\Z$ and $\sigma \in \NE(E)_\Z$, 
we denote by $q^d$ and $q^\sigma$ the corresponding 
elements in the group ring $\Q[\NE(X)_\Z]$ 
and $\Q[\NE(E)_\Z]$ respectively. 
We write: 
\[
q^\sigma = q_0^k  q^{d} \qquad \text{when} \quad 
\sigma = k [\sigma_0] + d 
\]
where $q_0 = q^{\sigma_0}$ is the variable
corresponding to the maximal section $\sigma_0$.
For $\sigma \in \NE(E)_\Z^{\rm sec}$, let
$\cM_{\rm S}(\sigma)$ denote the moduli space of
stable maps from genus-zero closed nodal Riemann
surfaces to $E$ in the class $\sigma$ with one
marked point whose image lies in a fixed fibre
$\iota\colon M\hookrightarrow  E$.  We can write
\[
\cM_{\rm S}(\sigma ) = \cM_1 (\sigma ) \times_E M 
\]
using the usual moduli space $\cM_1(\sigma)$ of
genus-zero one-pointed stable maps to $E$ in the
class $\sigma$.  Since $\cM_1(\sigma)$ has a
Kuranishi structure (without boundary) of virtual
real dimension $2n+2\lan c_1(E), \sigma\ran -2$
(with $n := \dim_\C M$) and we may assume that 
the evaluation map $\cM_1(\sigma) \to E$ is weakly
submersive (see \cite[Theorem 7.11]{Fukaya-Ono-Top99}), 
the fibre product $\cM_{\rm S}(\sigma)$ is equipped with 
the induced Kuranishi structure of virtual dimension:
\begin{equation} 
\label{eq:vdim-Seidel}
\vdim_\R  \cM_{\rm S}(\sigma) = 
2n + 2\lan c_1^{\rm vert}(E), \sigma \ran. 
\end{equation} 
Here $c_1^{\rm vert}(E)$ denotes the first Chern class 
of the vertical tangent bundle $T_{\rm vert} E$, 
\[
T_{\rm vert} E: = \Ker(d\pi \colon TE \to \pi^* T\P^1) 
\]
with $\pi \colon E\to \P^1$ the natural projection. 
(Note that $\lan c_1(E),\sigma \ran= \lan c_1^{\rm vert}(E), 
\sigma \ran +2$.)  
Let $\ev \colon \cM_{\rm S}(\sigma) \to M$ be the 
evaluation map 
and let $[\cM_{\rm S}(\sigma)]^{\rm vir}$ be  
the virtual fundamental cycle of $\cM_{\rm S}(\sigma)$. 

\begin{definition}\label{def:seidel}
The \emph{Seidel element} associated to the $\C^\times$-action 
on $M$ is the class
\begin{equation}\label{eq:seidel}
S:=\sum_{\sigma \in \NE(E)_\Z^{\sec}} 
\PD \left(  
\ev_*[\cM_S(\sigma)]^{\rm vir} \right) 
q^{\sigma} 
\end{equation}
in $H^*(M;\Q) \otimes \Q[\![\NE(E)_\Z]\!]$.  
Here $\PD$ stands for the Poincar\'{e} duality 
isomorphism. 
By \eqref{eq:split-sigma0}, we can factorize $S$ as 
$S = q_0 \tS$ with $\tS$ in the 
small quantum cohomology ring 
\[
QH(M) := H(M;\Q) \otimes \Q[\![\NE(M)_\Z]\!]
\]
and $q_0 := q^{\sigma_0}$ as above. 
Then $\tS$ is an invertible element of
$QH(M)[q^{-d}\!:\!d\in \NE(M)_\Z]$: 
the Seidel element $\tS'$ associated with the reverse $\C^\times$-action 
satisfies $\tS'\star \tS = q^{d_0}$ for some $d_0\in H_2(M,\Z)$  
\cite{Se-pi97,LaMc-To99,Mc-Qu00}.  
\end{definition}

\begin{remark} 
Using genus zero one-point Gromov-Witten invariants 
for $E$, we can write 
\[
S = \sum_{\sigma\in \NE(E)^{\rm sec}_\Z} 
\sum_{i} 
\lan \iota_* \phi_i \ran_{0,1,\sigma}^E \phi^i q^\sigma 
\]
where $\{\phi_i\}$ is a basis of $H^*(M;\Q)$,   
$\{\phi^i\}$ is the dual basis 
with respect to the Poincar\'{e} pairing 
and $\iota \colon M\to E$ is the inclusion of a fibre. 
(We followed the standard notation of Gromov-Witten 
invariants as in \cite{CoKa-Mi99}.) 
\end{remark} 

\begin{remark} 
\label{rem:Seidelrep} 
For a general symplectic manifold $M$, 
we use the Novikov ring $\Lambda$ 
\[
\Lambda := \left\{\textstyle\sum_{d \in H_2(M;\Z)} c_d q^d\, :\, 
c_d \in \Q, \ \sharp \{ d \,:\, c_d \neq 0, \ \lan \omega, d \ran \le E\} <\infty 
\text{ for all } E \in \R \right\}. 
\]
instead of $\Q[\![\NE(M)_\Z]\!]$. 
The Seidel elements associated to loops in $\Ham(M,\omega)$  
define a group homomorphism 
\cite{Se-pi97,LaMc-To99,Mc-Qu00}: 
\[
\pi_1(\Ham(M,\omega)) \to QH(M)_\Lambda^\times 
\Bigr/ \left \{q^{d} \,:\, d \in H_2(M;\Z) \right \} 
\] 
which is called the \emph{Seidel representation}, 
where $QH(M)_\Lambda = H^*(M;\Q) \otimes \Lambda$ 
denotes the quantum cohomology ring 
over $\Lambda$. 
\end{remark}

\section{Degeneration Formula}
Let $M$ be a smooth projective variety equipped with a 
$\C^\times$-action. We take an $S^1$-invariant 
K\"{a}hler form $\omega$ on $M$. 
Let $L$ be a Lagrangian submanifold 
of $M$ which is preserved by $S^1\subset \C^\times$, i.e.\ 
$\lambda L \subset L$ for $\lambda \in S^1$. 
Instead of counting holomorphic discs in $(M,L)$,
we shall consider the problem of counting
holomorphic \emph{disc sections} of the bundle
$M\times \D \to \D$ with boundary in $L\times
S^1$.  Then we degenerate the target $M\times \D$
into the union of the associated bundle $E$ and
$M\times \D$.  From this we expect a certain
relationship between Seidel elements and disc
counting invariants.  We assume that $M$ is a
smooth projective variety with a
$\C^\times$-action for simplicity, but the
degeneration formula in this section makes sense
for a symplectic manifold with a Hamiltonian
circle action (or a loop in the group of
Hamiltonian diffeomorphisms) in general.

\subsection{Degeneration of $M\times \D$} 
\label{ss:def} 
Let $\D$ denote the unit disc $\{z\in \C\,:\,
|z|\le 1\}$.  A degeneration of the disc $\D$ into
the union $\D \cup \P^1$ is given by the blowup
$\Bl_{(0,0)}(\D\times \C)$ of $\D \times \C$ at
the origin.  The projection $\pi\colon \Bl_{(0,0)}
(\D \times \C) \to \C$ satisfies $\pi^{-1}(t)
\cong \D$ for $t\neq 0$ and $\pi^{-1}(0) \cong \D
\cup \P^1$.  Explicitly:
\[
\Bl_{(0,0)} (\D \times \C) 
= \left \{ (z, t, [\alpha,\beta] )\in \D \times \C\times 
\P^1 \, :\, 
z \beta - t \alpha =0 \right\}.  
\]
An $M$-bundle $\cE$ over $\Bl_{(0,0)}(\D\times \C)$ 
is defined as follows. 
\[
\cE:= \left\{(x, z, t, (\alpha,\beta)) \in 
M \times \D \times \C \times (\C^2\setminus \{0\}) \,:\, 
z\beta - t\alpha =0\right\}\big/ \C^\times 
\]
where $\C^\times$ acts as 
$(x,z,t,(\alpha,\beta)) \mapsto 
(\lambda x, z, t, (\lambda \alpha, \lambda \beta))$.    
We have a natural projection $\pi\colon \cE \to \C$. 
One can see that 
\begin{equation} 
\label{eq:degeneration-fibre} 
\cE_t = \pi^{-1}(t) = 
\begin{cases} 
M \times \D  & \text{if } t\neq 0; \\
E \cup_{M} (M\times \D) & \text{if } t=0 
\end{cases} 
\end{equation} 
where $E$ is the associated bundle (Definition
\ref{def:a-bundle}) of the $\C^\times$-action on
$M$.  One can also construct $\cE$ as a symplectic
quotient:
\[
\cE = \left\{ (x,z,t,(\alpha,\beta)) \, :\, 
z\beta - t\alpha =0, \ H(x) + |\alpha|^2+|\beta|^2= c 
\right \} 
\big/S^1 
\]
where $H\colon M\to \R$ is the moment map of 
the $S^1$-action and $c>\max_{x\in M} H(x)$ is a real number. 
We can equip $\cE$ with a symplectic structure. 
The boundary $\partial \cE_t$ can be identified with 
$M\times S^1$ via the map:
\begin{equation} 
\label{eq:bdry-trivialization} 
M\times S^1 \ni (x,z) \mapsto [x,z,t,(z,t)] \in \partial\cE_t. 
\end{equation} 
Via this identification, $\cE_t$ contains 
a Lagrangian submanifold $\hL_t := L\times S^1$ 
in the boundary $M\times S^1 \cong \partial \cE_t$. 

We can close $\cE_t$ by attaching $M\times
\D$ to the boundary for each $t$ and get a
degenerating family $\ocE$ of closed manifolds.
More explicitly, we define:
\[
\ocE = \left\{(x, (z,w), t, (\alpha,\beta)) \in
  M\times (\C^2\setminus \{0\}) \times \C \times
  (\C^2 \setminus \{0\}) \,:\, t \alpha w = z
  \beta\right\}\big/\C^\times \times \C^\times
\]
where $\C^\times \times \C^\times$ acts as 
\[
(x,(z,w),t,(\alpha,\beta)) \mapsto (\lambda_1^{-1}
\lambda_2 x, (\lambda_1z, \lambda_1 w), t,
(\lambda_2 \alpha, \lambda_2 \beta) ).
\]
This is an $M$-bundle over 
\[
  \Bl_{(0,0)} (\P^1\times \C) = \left\{([z,w], t,
  [\alpha,\beta]) \in \P^1 \times \C \times \P^1
  \,:\, t \alpha w = z\beta \right \}.
\]
With respect to the projection $\pi \colon \ocE
\to \C$ to the $t$-plane, we have
\[
\ocE_t = \pi^{-1}(t) = 
\begin{cases} 
M\times \P^1 & \text{if } t\neq 0; \\
E \cup_{M} E' & \text{if } t=0. 
\end{cases} 
\]
where $E'$ is the associated bundle of the
$\C^\times$-action on $M$ \emph{inverse} to the
original one.  Note that $\cE$ is contained in
$\ocE$ as the locus $\{w=1, \, |z|\le 1\}$ and
$\ocE = \cE \cup_{M\times S^1\times \C} (M\times
\D^2 \times \C)$.  We can also equip $\ocE$ with a
symplectic structure by describing it as a
symplectic quotient in a similar manner.

A topological description is given as follows. We
start from a trivial $M$-bundle $M\times \P^1$
over $\P^1$.  We cut $\P^1$ into 3 pieces: $\P^1=
\D_0 \cup A \cup \D_\infty$, where $\D_0 =
\{|z|\le 1/2\}$, $A=\{1/2\le |z|\le 2\}$ and
$\D_\infty = \{|z|\ge 2\} \cup \{\infty\}$.  One
can twist the clutching function along $\partial
\D_0$ and $\partial \D_\infty$ by the given
$S^1$-action on $M$; namely
\begin{equation} 
\label{eq:clutchingtwist}
M\times \P^1 = (M\times \D_0) \cup_{g_1} (M\times
A) \cup_{g_2} (M\times \D_\infty)  
\end{equation} 
where the clutching functions $g_1$, $g_2$ are
given respectively by
\begin{align*} 
  & g_1 \colon M\times \partial \D_0 \ni (x,
  \tfrac{1}{2}e^{i\theta}) \longmapsto
  (e^{-i\theta} x, \tfrac{1}{2} e^{i\theta}) \in M
  \times \partial_0A
  \\
  & g_2 \colon M\times \partial_\infty A \ni (x, 2
  e^{i\theta}) \longmapsto (e^{i\theta} x, 2
  e^{i\theta} ) \in M\times \partial \D_\infty
\end{align*} 
where we set $\partial A =\partial_0 A \cup \partial_\infty A$. 
Collapsing $M\times S^1 \subset M\times A$ 
down to $M$, we get the singular central fibre $E \cup_M E'$. 
In fact, for $|t|<1$, one can decompose $\ocE_t$ as 
\begin{align*} 
\ocE_t = & \left\{[x,(tz,1), t, (z,1)]\,:\, |z|\le 1\right\} \\
& \cup \left\{
[x,(z,1),t,(1,\beta)]\,:\, t  = \beta z, \ 
 |\beta|\le 1, \ |z|\le 1\right\} \\
& \cup \left\{[x,(1,w),t,(1,wt)]\,:\, |w|\le 1
\right\}. 
\end{align*} 
This corresponds to the decomposition 
\eqref{eq:clutchingtwist} of $M\times \P^1$ above. 

\begin{remark} 
\label{rem:closing}
We shall consider stable holomorphic discs in
$(\ocE_t, \hL_t)$ which project onto the
holomorphic disc $(\D^2, S^1) \subset (\P^1,S^1)$.
Such stable holomorphic discs are entirely
contained in the half-space $\cE_t$ of $\ocE_t$,
so the choice of ``closing" of $\cE_t$ is not 
relevant. 
\end{remark} 

\begin{remark} 
We can perform a similar construction for a general symplectic 
manifold $(M,\omega)$ equipped with a Lagrangian submanifold $L$ 
and a loop $\{\phi_\theta\}_{\theta \in [0,2\pi]}$ 
in the group $\Ham(M,\omega)$ of Hamiltonian diffeomorphisms  
such that $\phi_\theta(L) = L$ for all $\theta$. 
We can twist the clutching function of the trivial 
$M$-bundle $M\times \P^1$ as in \eqref{eq:clutchingtwist}
where $g_1$, $g_2$ there are replaced with 
\[
g_1(x, \tfrac{1}{2} e^{i\theta}) = (\phi_{-\theta}(x), \tfrac{1}{2} e^{i\theta}), 
\quad 
g_2(x, 2 e^{i\theta}) = (\phi_{\theta} (x), 2 e^{i\theta}). 
\]
Then we can degenerate the annulus $A$ into the union of two discs 
(in a one-parameter family) in the middle part $M\times A$.  
In the degeneration family, we have a family of Lagrangian 
submanifolds $L\times S^1$ 
lying in the boundary of $M\times \D_0 \cup_{g_1} M\times A$.  
\end{remark}

\subsection{Relative homology classes of degenerating discs} 
\label{subsec:relativehomology} 
We write $\cL=\bigcup_{t\in \C} \hL_t$. 
The total space $(\cE, \cL)$ of the family has a deformation retraction 
to the central fibre $(\cE_0,\hL_0)$. 
This gives a retraction map for $t\neq 0$: 
\[
r\colon H_2(\cE_t, \hL_t) \longrightarrow 
H_2(\cE, \cL)  \cong H_2(\cE_0, \hL_0). 
\] 
Let $\pi \colon \cE \to \Bl_{(0,0)}(\D\times \C)$ denote the 
natural projection. 
We have the following commutative diagram: 
\[
\begin{CD} 
H_2(\cE_t, \hL_t) @>{r}>> H_2(\cE_0,\hL_0) \\ 
@V{\pi_*}VV  @V{\pi_*}VV \\ 
H_2(\D, S^1) @>{r}>> H_2(\P^1 \cup \D, S^1).
\end{CD} 
\]
Under the natural identifications $H_2(\D,S^1;\Z)
\cong \Z$ and $H_2(\P^1 \cup \D, S^1;\Z) \cong
H_2(\P^1;\Z) \oplus H_2(\D,S^1;\Z) \cong \Z^2$,
the bottom arrow is given by $n\mapsto (n,n)$.  We
are interested in \emph{section classes} lying in
the following groups:
\[
H_2^{\rm sec}(\cE_t, \hL_t) = \pi_*^{-1}(1),
\text{ for } t\neq 0, \text{ and } 
H_2^{\rm sec}(\cE_0, \hL_0) = \pi_*^{-1}(1,1).
\] 
There is an induced retraction map $r\colon
H_2^{\rm sec}(\cE_t,\hL_t)  
\to H_2^{\rm sec}(\cE_0,\hL_0)$ for $t\neq 0$. 
\begin{lemma} 
\label{lem:relativehomology} 
Assume that $M$ is simply connected and $L$ is connected. 
Then we have 
\begin{align}
\label{eq:relativehomology} 
\begin{split}  
& H_2^{\rm sec}(\cE_t,\hL_t) \cong H_2(M,L) \quad 
\text{for } t\neq 0 \\
& H_2^{\rm sec}(\cE_0,\hL_0) \cong H_2^{\rm sec}(E)\times_{H_2(M)} H_2(M,L)
\end{split} 
\end{align} 
\end{lemma} 
\begin{proof} 
Recall that $(\cE_t,\hL_t) \cong (M\times \D, L \times S^1)$ 
for $t\neq 0$. 
We show the isomorphism: 
\[
(p_{1*}, p_{2*}) \colon 
H_2(M\times \D, L\times S^1) \cong 
H_2(M,L) \times H_2(\D, S^1)
\]
where $p_1, p_2$ are natural projections. 
Because we have sections $i_1, i_2 \colon (M,L) \to (M\times \D, L\times S^1)$
such that $p_1 \circ i_1 = \id$, $p_2 \circ i_2=\id$, 
$p_2\circ i_1 = \text{const}$ 
and $p_1 \circ i_2 = \text{const}$, the map $(p_{1*}, p_{2*})$ 
is surjective.  To show that it is injective, we use the 
commutative diagram: 
\[
\begin{CD} 
0 @>>> 0 @>>> H_2(\D,S^1) @>>> H_1(S^1) \\
@AAA @AAA @A{p_{2*}}AA @AAA \\ 
H_2(L\times S^1) @>>> H_2(M\times \D) 
@>>> H_2(M\times \D, L\times S^1) 
@>>> H_1(L\times S^1) \\ 
@V{\text{epi}}VV @V{\cong}VV @V{p_{1*}}VV @VVV \\ 
H_2(L) @>>> H_2(M) @>>> H_2(M,L) @>>> H_1(L).  
\end{CD} 
\]
Here all the horizontal sequences are exact. 
The injectivity of $(p_{1*}, p_{2*})$ follows from the 
diagram chasing and $H_1(L\times S^1) 
\cong H_1(L) \oplus H_1(S^1)$ (here we use the 
condition that $L$ is connected). 
Then $H_2^{\rm sec}(\cE_t, \hL_t) \cong H_2(M,L)$ for $t\neq 0$
follows. 
 
The Mayer-Vietoris exact sequence for $\cE_0 = E \cup_M (M\times \D)$ 
gives 
\[
\begin{CD} 
H_2(M) @>>> H_2(E) \oplus H_2(M\times \D, L\times S^1) 
@>>> H_2(\cE_0,\hL_0) 
@>>> 0. 
\end{CD} 
\]
Here we used $H_1(M) =0$. 
The formula for $H_2^{\rm sec}(\cE_0, \hL_0)$ follows. 
\end{proof} 

Henceforth we assume that \emph{$L$ is connected 
and $M$ is simply-connected}.  
\begin{remark} 
\label{rem:closing-homology} 
The natural map $H_2(\cE_t,\hL_t) \to H_2(\ocE_t,\hL_t)$ is 
injective because the composition: 
\[
H_2(M\times \D, L\times S^1) \to H_2(M\times \P^1, L\times S^1) 
\xrightarrow{(p_{1*}, p_{2*})} H_2(M,L) \oplus H_2(\P^1,S^1) 
\]
is injective. 
\end{remark} 
\begin{notation} 
\label{nota:discsection} 
We denote by 
\begin{align*} 
  \hbeta & \in H_2^{\rm sec}(\cE_t,\hL_t) \cong
  H_2^{\rm sec}(M\times \D,
  L\times S^1) \quad (t\neq 0) \\
  \sigma + \hbeta & \in H_2^{\rm sec}(\cE_0,\hL_0)
\end{align*} 
the homology classes corresponding to $\beta \in
H_2(M,L)$ and to $[\sigma, \beta] \in H_2^{\rm
  sec}(E) \times_{H_2(M)}H_2(M,L)$ respectively,
under the isomorphism \eqref{eq:relativehomology}
in Lemma \ref{lem:relativehomology}.
\end{notation} 

Let $u\colon \D \to M$ be a map such that
$u(e^{i\theta}) = e^{i\theta}\cdot u(1)$, namely,
$u$ is a disc contracting an $S^1$-orbit in $M$.
This defines a section $\sigma(u)$ of the
associated bundle $E \to \P^1$:
\begin{alignat*}{7} 
  & \sigma(u) |_{\D_0} &\colon& \D_0 \to
  E|_{\D_0} &\cong& M\times \D_0, \quad
  & & z \mapsto (u(1),z) \\
  & \sigma(u)|_{\D_\infty} &\colon& \D_\infty \to
  E|_{\D_\infty} &\cong& M\times \D_\infty,\quad & &
  z\mapsto (u(z^{-1}),z)
\end{alignat*} 
where $\D_0= \{z\in \C\,:\, |z|\le 1\}$ and
$\D_\infty = \{ z\in \C\, : \, |z|\ge 1\} \cup
\{\infty\}$; here we used the gluing construction
of $E$ in Remark \ref{rk:ssidel}.

Recall the maximal section class $\sigma_0$ of $E$
in \S \ref{subsec:Seidel}.  We introduce a similar
\emph{maximal disc class} $\alpha_0 \in H_2(M,L)$
as follows. Take a path $\gamma \colon [0,1] \to
M$ such that $\gamma(0) \in F_{\rm max}$ and
$\gamma(1) \in L$, where $F_{\rm max}$ is the
maximal fixed component.  We define $\alpha_0$ to
be the class represented by the disc $\D \ni r
e^{i\theta} \mapsto e^{-i \theta}\cdot \gamma(r)
\in M$.  The homotopy class here is independent of
the choice of a path $\gamma$ because $M$ is
simply-connected and $L$ is connected.  The
boundary of $\alpha_0$ is an inverse $S^1$-orbit
on $L$.

\begin{proposition}
\label{prop:retractionmap} 
The retraction map $r\colon H_2^{\rm sec}(\cE_t,\hL_t)\to 
H_2^{\rm sec}(\cE_0,\hL_0)$ (for $t\neq 0$) 
of section classes is an isomorphism.   
It is given by (under Notation \ref{nota:discsection})   
\begin{align*} 
r(\hbeta) = \sigma(u) -\hat{u} + \hbeta 
= \sigma_0 + \halpha_0 + \hbeta  
\qquad 
\text{for }\beta \in H_2(M,L) 
\end{align*} 
where $u\colon \D \to M$ is an 
arbitrary disc whose boundary is an $S^1$-orbit in $L$,  
$\sigma_0$ is the maximal section and $\alpha_0$ 
is the maximal disc. 
In particular we have the commutative diagram
\[
\begin{CD} 
  H_2^{\rm sec}(\cE_t, \hL_t) @>{\cong}>{r}>
  H_2^{\rm sec}(\cE_0,\hL_0) \\ 
  @V{\partial}VV  @V{\partial}VV \\
  H_1(L) @>{\cong}>{-\lambda}> H_1(L)
\end{CD}
\]
where the bottom map is the subtraction of the
class $\lambda=[\partial u]$ of an $S^1$-orbit on
$L$.
\end{proposition} 
\begin{proof} 
  Consider a constant section $s_{\text{triv}}(z)
  = (x,z)$ of $M\times \D \cong \cE_t$ with $x\in
  L$.  By the topological description of the
  degeneration given in \S \ref{ss:def}, we see
  that $s_{\text{triv}}$ can degenerate to the
  union:
  \[
  \sigma(u) \cup \tilde{u} \colon \P^1 \cup \D \to E
  \cup_M (M\times \D)
  \]
  where $u\colon \D \to M$ is a disc contracting
  the $S^1$-orbit $e^{i\theta} x$ on $L$ and
  $\tilde{u} \colon \D \to M\times \D$ is given by
  $z\mapsto (u(\overline{z}), z)$.  This shows
  that $r([s_{\text{triv}}])= \sigma(u) -
  \hat{u}$.  Since the retraction map is a
  homomorphism of $H_2(M,L)$-modules, we have
  $r(\hbeta) = \sigma(u) - \hat{u}+ \hbeta$ in
  general.  When $u$ is a disc of the form: $\D
  \ni r e^{i\theta} \mapsto e^{i\theta} \cdot
  \gamma(r) \in M$, where $\gamma \colon [0,1]\to
  M$ is a path such that $\gamma(0) \in F_{\rm
    max}$ and $\gamma(1) \in L$, $\sigma(u)$ is
  homotopic to the maximal section $\sigma_0$ and
  $[u] = -\alpha_0$. This shows the formula
  $r(\hbeta) = \sigma_0 + \hat{\alpha}_0+\hbeta$.
  It is easy to check that $r$ is an isomorphism
  between section classes.
\end{proof} 
\begin{remark}
\label{rem:trivialization-difference}
The latter statement in Proposition
\ref{prop:retractionmap} 
is a consequence of the
difference of trivializations of $\partial \cE_t$
($t\neq 0$) and $\partial \cE_0$.  Recall that we
have a trivialization $\partial \cE_t \cong
M\times S^1$ in \eqref{eq:bdry-trivialization}
depending smoothly on $t\in \C$. For $t\neq 0$,
this trivialization is induced from the
isomorphism $\cE_t \cong M\times \D$ in
\eqref{eq:degeneration-fibre}; however for $t=0$,
this trivialization differs by the $S^1$-action
from the one induced by the isomorphism $\cE_0
\cong E \cup_M (M\times \D)$ in
\eqref{eq:degeneration-fibre}.
\end{remark} 


\begin{lemma}[Maslov index and vertical Chern number] 
\label{lem:Maslov}  
Let $u\colon \D \to M$ be a disc with boundary an
$S^1$-orbit on $L$, i.e.\ $u(e^{i\theta}) =
e^{i\theta} \cdot u(1)$ and $u(1)\in L$. Then $u$
defines a class in $\pi_2(M,L)$ and we have $
\mu(u) = 2 \lan c_1^{\rm vert}(E),
[\sigma(u)]\ran$.
\end{lemma} 
\begin{proof} 
  We recall the definition of Maslov index of a
  disc $u\colon (\D,S^1) \to (M,L)$.  We set
  $\gamma = u|_{\partial \D}$.  Note that
  $u^*TM|_{S^1}$ is a complexfication of the
  subbundle $\gamma^*TL$.  Thus
  $\det(u^*TM)|_{S^1}$ is a complexification of
  the real line bundle $\det_\R(\gamma^*TL)$.  On
  the other hand $\det_\R (\gamma^*TL)^{\otimes
    2}$ has a canonical orientation.  Take a
  positive (nowhere vanishing) section $s_0$ of
  $\det_\R(\gamma^* TL)^{\otimes 2}$.  The Maslov
  index of $u$ is the signed count of zeros of a
  transverse section $s$ of $\det(u^* TM)^{\otimes
    2}$ such that $s|_{\partial \D} = s_0$.

  When $u|_{\partial \D}$ is an $S^1$-orbit of
  $L$, we can take $s_0$ above to be
  $S^1$-equivariant.  A transverse section $s$ of
  $\det(u^*TM)^{\otimes 2}$ with $s|_{\partial \D}
  = s_0$ defines a section $t\in \det(\sigma(u)^*
  T_{\rm vert} E)^{\otimes 2}$ by
\[
t|_{\D_0} (z)  = s_0(1), \quad 
t|_{\D_\infty}(z)  = s(z^{-1}).  
\]
Then the numbers of zeros of $t$ and $s$ coincide.
The lemma follows.
\end{proof} 
Proposition \ref{prop:retractionmap} and Lemma
\ref{lem:Maslov} show the following corollary:
\begin{corollary} 
\label{cor:Maslov} 
Let $r\colon H_2^{\rm sec}(\cE_t,\hL_t) \to
H_2^{\rm sec}(\cE_0,\hL_0)$ be the retraction map
for $t\neq 0$.  Suppose that $r(\hbeta)= \sigma +
\halpha$ with $\alpha,\beta \in H_2(M,L)$.  Then
$\mu(\beta) = 2 \lan c_1^{\rm vert}(E), \sigma
\ran + \mu(\alpha)$.
\end{corollary} 

\begin{remark} 
We have $\mu(\hbeta) = \mu(\beta) +2$ for $\beta \in H_2(M,L)$. 
\end{remark} 

\subsubsection{Example} 
We give an example of degenerating holomorphic
discs.  Consider a family of (constant)
holomorphic disc sections $u_t\colon (\D,S^1) \to
(\cE_t, \hL_t)$ given by
\[
u_t(z) = [x_0, z, t, (z,t)] 
\]
for some $x_0\in L$. 
For a fixed non-zero $z\in \D$, we have 
\[
\varphi (z) := \lim_{t\to 0} u_t(z) = [x_0, z, 0, (z,0)]. 
\]
This can be completed to a holomorphic disc
section $\varphi \colon \D \to M\times \D \subset
\cE_0$.  Note that the limit
\[
\lim_{z\to 0} 
\varphi(z) = [x_1, 0, 0, (1,0)]
\qquad \text{where} \quad x_1 := \lim_{z\to 0}
z^{-1} x_0 \in M,
\]
exists by the completeness of $M$, and it is fixed by the
\(\C^\times\) action.
On the other
hand, we can see a bubbling off holomorphic sphere
at $z=0$ by the usual rescaling:
\[
\psi(z):= \lim_{t\to 0} u_t( tz) 
=\lim_{t\to 0} [t^{-1} x_0, tz, t, (z,1)]
=[x_1, 0,0, (z,1)]. 
\]
This defines a holomorphic section $\psi \colon
\P^1 \to E \subset \cE_0$ associated to the
$\C^\times$-fixed point $x_1\in M$.  Note that
$\psi(\infty) = \varphi(0)$ and $\partial \varphi$
is an inverse $S^1$-orbit on $L$.

\subsection{Degeneration formula} 
\label{subsec:degeneration} 
In what follows, we propose a conjectural
degeneration formula and discuss its consequences.
As before, $M$ denotes a smooth projective variety
equipped with a $\C^\times$-action and an
$S^1$-invariant K\"{a}hler form $\omega$; $L$ is a
Lagrangian submanifold which is preserved by the
$S^1$-action.  We assume that $M$ is
simply-connected and $L$ is connected.  Moreover
we assume that $L$ is oriented and relatively spin
and we fix a relative spin structure
\cite[Definition 8.1.2]{FuOh-La09}.

Take $\beta \in H_2(M,L)$.  We consider the moduli
space $\cM_1(\hbeta)$ of stable holomorphic maps
from genus zero bordered Riemann surface
$(\Sigma, \partial \Sigma)$ to $(\ocE_t,
\hL_t)\cong (M\times \P^1, L\times S^1)$ with one
boundary marked point and in the class $\hbeta \in
H_2^{\rm sec}(\cE_t, \hL_t)$ (where $t\neq 0$; see
Notation \ref{nota:discsection}).  Such stable
maps project onto the disc $(\D,S^1) \subset
(\P^1,S^1)$ on the base and so are contained in
$\cE_t$ (see Remarks \ref{rem:closing} and
\ref{rem:closing-homology}).  The virtual
dimension of $\cM_1(\hbeta)$ is $n+1+\mu(\hbeta)
-2 = n+1+\mu(\beta)$ with $n:= \dim_\C M$.  The
corresponding moduli space at $t=0$ should be
described as the fibre product:
\[
\bigcup_{r(\hbeta) = \sigma + \halpha } \cM_{\rm
  S}(\sigma) \times_M \cM_{1,1}^{\rm rel}(\halpha)
\] 
where $\cM_{\rm S}(\sigma)$ is the moduli space of
holomorphic sections of $E$ appearing in \S
\ref{subsec:Seidel} and $\cM_{1,1}^{\rm
  rel}(\halpha)$ is the moduli space of stable
holomorphic maps from genus zero bordered Riemann
surfaces to $(M\times \P^1, L\times S^1)$ in the
class $\halpha \in H_2^{\rm sec}(M\times \D,
L\times S^1)$ with one boundary marked point and
one interior marked point such that the image of
the interior marked point lies in $M\times \{0\}$.
The superscript ``rel" (which means ``relative")
signifies the last condition.  The fibre product
above is taken with respect to the interior
evaluation maps.  One can write:
\begin{equation} 
\label{eq:M11rel} 
\cM_{1,1}^{\rm rel}(\halpha) = 
\cM_{1,1}(\halpha) \times_{M\times \P^1}
(M\times \{0\}) 
\end{equation} 
using the moduli space $\cM_{1,1}(\halpha)$ of
bordered stable maps to $(M\times \P^1, L\times
S^1)$ of class $\halpha$ with one boundary marking
and one interior marking.  Then a Kuranishi
structure on $\cM_{1,1}^{\rm rel}(\halpha)$ is
induced from the Kuranishi structure on
$\cM_{1,1}(\halpha)$ (as defined in \cite[\S
7.1]{FuOh-La09}) via this presentation.  The
virtual dimension is
\[
\vdim \cM_{1,1}^{\rm rel}(\halpha) 
= n+1 + \mu(\alpha).  
\]
We write $\ev^{\rm (i)} \colon \cM^{\rm
  rel}_{1,1}(\halpha) \to M$ for the interior
evaluation map and $\ev^{\rm (b)} \colon \cM^{\rm
  rel}_{1,1}(\halpha) \to L\times S^1$ for the
boundary evaluation map.

When the virtual fundamental \emph{chains} on the
moduli spaces $\cM_1(\hbeta)$ and $\cM_{\rm
  S}(\sigma) \times_M \cM_{1,1}^{\rm
  rel}(\halpha)$ happen to be \emph{cycles}, we
expect the following degeneration formula:
\begin{equation} 
\label{eq:degenerationformula} 
\varphi_* \ev_*[\cM_1(\hbeta) ]^{\rm vir} = 
\sum_{r(\hbeta) = \sigma + \halpha} 
\ev_* [\cM_{\rm S} (\sigma)\times_M \cM^{\rm
  rel}_{1,1}(\halpha)]^{\rm vir}  
\end{equation} 
in $H_*(L\times S^1)$.  Here $\ev$ on the
both-hand sides denotes the evaluation map at the
boundary markings taking values in $\hL_t \cong
L\times S^1$ and $\varphi \colon L\times S^1 \to
L\times S^1$ is the map $(x,e^{i\theta}) \mapsto
(e^{-i\theta} \cdot x, e^{i\theta})$ which
corresponds to the difference of boundary
trivializations (see Remark
\ref{rem:trivialization-difference}).  We will
study below when the both-hand sides of
\eqref{eq:degenerationformula} make sense as
cycles; then will calculate them in terms of
Seidel elements and open Gromov-Witten invariants.

\subsubsection{The left-hand side of
  \eqref{eq:degenerationformula}}
\label{subsubsec:lhs} 
When $\beta=0$, $\cM_1(\hbeta)$ consists of
constant disc sections and $\ev\colon
\cM_1(\hbeta) \to L\times S^1$ is a homeomorphism.
All constant disc sections are Fredholm regular.
When $\beta \neq 0$, we have a natural map
\[
\cM_1(\hbeta) \to \cM_1(\beta) 
\]
induced by the projection $\cE_t \to M$, where
$\cM_1(\beta)$ is the moduli space of one-pointed
bordered stable maps to $(M,L)$ in the class
$\beta$.  By taking the graph of a disc component,
we can see that this map is surjective.
Therefore, for $\beta \neq 0$, $\cM_1(\hbeta)$ is
non-empty if and only if $\cM_1(\beta)$ is
non-empty.  Moreover, if $\cM_1(\beta)$ is
non-empty, $\cM_1(\hbeta)$ has boundary 
(see \cite[\S 7.1.1]{FuOh-La09} for the boundary description) 
since a bordered stable map of class $\hbeta$ can be
constructed
as the union of a constant disc section
and a disc of class $\beta$ (which is constant in the $\D$-direction).  
Therefore, we have
\begin{lemma} 
\label{lem:lhs} 
The virtual cycle $\ev_*[\cM_1(\hbeta)]^{\rm vir}$ 
is well-defined if $\cM_1(\beta) = \emptyset$. We have 
\[
\varphi_*\ev_*[\cM_1(\hbeta)]^{\rm vir} = 
\begin{cases}
  [L\times S^1]  & \text{if } \beta =0; \\
  0 & \text{if } \beta \neq 0 \text{ and }
  \cM_1(\beta) = \emptyset.
\end{cases}
\]
\end{lemma}

\subsubsection{The right-hand side of
  \eqref{eq:degenerationformula}}  
Take $(\sigma, \alpha)\in H_2^{\rm sec}(E) \times
H_2(M,L)$ such that $\sigma + \halpha =
r(\hbeta)$.  By Corollary \ref{cor:Maslov} and
Proposition \ref{prop:retractionmap}, we have
\begin{align}
\label{eq:Maslov+vChern} 
& \mu(\beta) = 
2 \lan c_1^{\rm vert}(E), \sigma \ran + \mu(\alpha) 
\\ 
\label{eq:boundaryrelation}
& \partial \beta = \partial \alpha +\lambda 
\end{align} 
where $\lambda\in H_1(L)$ is the class of an 
$S^1$-orbit. 

Suppose $\alpha=0$. 
This can happen only when $\partial \beta = \lambda$ 
by \eqref{eq:boundaryrelation}. 
Since $\alpha=0$, $\cM_{1,1}^{\rm rel}(\halpha)$
consists of constant disc sections and $\ev^{\rm
  (b)} \colon \cM_{1,1}^{\rm rel}(\halpha) \cong L
\times S^1$.  The interior evaluation $\ev^{\rm
  (i)} \colon \cM_{1,1}^{\rm rel}(\halpha) \to M$
is given by the projection $L\times S^1 \to L
\subset M$.  Thus
\begin{align}
\label{eq:alphazero}
\begin{split}  
  \ev_* [\cM_{\rm S}(\sigma) \times_M \cM^{\rm
    rel}_{1,1}(\halpha)] ^{\rm vir}
  & = \ev_* [\cM_{\rm S}(\sigma)\times_M (L\times
  S^1)]^{\rm vir} \\ 
  & = (\cS_{\sigma} \cap [L]) \times [S^1]
\end{split} 
\end{align} 
where 
\begin{equation} 
\label{eq:cS}
\cS_{\sigma} := \PD  \left (
  \ev_*[\cM_{\rm S}(\sigma)]^{\rm vir} \right) 
\in H^{-\mu(\beta)}(M).  
\end{equation} 
Here we used the virtual dimension formula
\eqref{eq:vdim-Seidel} and
\eqref{eq:Maslov+vChern}.

Suppose $\alpha\neq 0$.  By the same argument as
in \S \ref{subsubsec:lhs}, $\cM_{1,1}^{\rm
  rel}(\halpha)$ is non-empty if and only if
$\cM_{1}(\alpha)$ is non-empty; also
$\cM_{1,1}^{\rm rel}(\halpha)$ has boundary if
$\cM_1(\alpha)$ is non-empty.  Assume that
$\cM_1(\alpha)$ has no boundary.  This means that
every stable map in $\cM_1(\alpha)$ has only one
disc component\footnote {See \cite[\S
  7.1.1]{FuOh-La09} for the boundary description
  of the moduli spaces.} (but possibly with sphere
bubbles).  Let us study the moduli space
$\cM_{1,1}^{\rm rel}(\halpha)$ and its boundary.
Since $\alpha \neq 0$, we have a map
\begin{equation}\label{eq:forget-udisc}
  \frf =(\frf_1, \frf_2) \colon \cM_{1,1}^{\rm
    rel}(\halpha) \to \cM_1(\alpha) \times S^1.   
\end{equation}
The first factor $\frf_1$ is given by projecting
bordered stable maps to $M$, forgetting the
interior marking and collapsing unstable
components; the second factor $\frf_2$ is the
boundary evaluation $\ev^{\rm (b)} \colon
\cM_{1,1}^{\rm rel}(\halpha) \to L\times S^1$
followed by the projection $L\times S^1 \to
S^1$. The map $\frf$ can be viewed as a
tautological family of stable discs over
$\cM_1(\alpha)\times S^1$. In fact we have the
following result.

\begin{lemma} 
\label{lem:universaldisc}
Let $u\colon (\Sigma, \partial \Sigma) \to (M,L)$
be a one-pointed bordered stable map of class
$\alpha$ and $x\in \partial \Sigma$ be the
boundary marking.  Suppose that $\Sigma$ has only
one disc component.  Then the fibre
$\frf^{-1}([u,\Sigma, x],z)$ at $([u,\Sigma,x],z)
\in \cM_1(\alpha) \times S^1$ can be identified
with the oriented real blow-up $\hSigma$ of
$\Sigma$ at $x$ (see the proof below for the
definition of $\hSigma$) and the interior
evaluation $\ev^{\rm (i)}$ on
$\frf^{-1}([u,\Sigma, x],z)$ can be identified
with the map $\hSigma \to \Sigma \xrightarrow{u}
M$.
\end{lemma} 
\begin{proof} 
  The assumption that $\Sigma$ has only one disc
  component was made for simplicity's sake (and
  this is the case we are interested in). In
  general, the fibre of $\frf$ can be identified
  with a smoothing of $\hSigma$ at the boundary
  singularities.  See \cite[Lemma
  7.1.45]{FuOh-La09} for a similar statement.

  We identify a neighbourhood of $x\in \Sigma$
  with the upper-half disc $\D_{+} = \{w\in
  \D\,:\,\Im(w)\ge 0\}$ where $x$ corresponds to
  $0\in \D_+$.  The oriented real blow-up
  $\hSigma$ is defined by replacing this
  neighbourhood with $[0,\pi] \times [0,1]$:
  \[
  \hSigma = (\Sigma\setminus \{x\})
  \cup_{\D_+\setminus \{0\}} \left ([0,\pi] \times
    [0,1]\right )
  \]
  where $\D_+\setminus \{0\}$ is identified with
  $[0,\pi] \times (0,1]$ by the map $w \mapsto
  (\arg(w), |w|)$.  Note that $\hSigma$ is a real
  analytic manifold (with boundary and corner)
  equipped with a natural projection $\hSigma \to
  \Sigma$.

  For a point $p\in \hSigma$, we shall construct a
  bordered stable map in the fibre
  $\frf^{-1}((u,\Sigma,x), z)$.  Suppose $p \in
  \hSigma \setminus \partial \hSigma \cong \Sigma
  \setminus \partial \Sigma$.  Note that $\Sigma$
  is a union of one disc component $\Sigma_0$ and
  trees of sphere bubbles.  If $p$ is in a tree of
  spheres bubbles, let $q$ be the intersection
  point of the tree (on which $p$ lies) and the
  disc $\Sigma_0$.  If $p$ is in the interior of
  $\Sigma_0$, set $q := p$.  Take a unique
  holomorphic map $v \colon \Sigma_0 \to \D$ which
  sends $q$ to $0\in \D$ and $x\in \partial
  \Sigma_0$ to $z\in S^1$.  Extend $v$ to the
  whole $\Sigma$ so that it is constant on each
  sphere component.  Then we obtain a bordered
  stable map $\hu = (u,v) \colon \Sigma \to
  M\times \D$ of class $\halpha$ with $p$ a new
  interior marked point.  (If $p$ is a node, we
  insert at the node a trivial sphere with an
  interior marking.)

  Next consider the case $p \in \partial \hSigma$.
  In this case, the corresponding bordered stable
  map is in the boundary of $\cM_{1,1}^{\rm
    rel}(\halpha)$.  See Figure
  \ref{fig:boundarypoints} below.  If $p$ is not
  in the exceptional locus $[0,\pi]$ of $\hSigma
  \to \Sigma$, we attach a disc $\D$ to $\Sigma$
  by identifying $1\in \partial\D$ with
  $p\in \partial \Sigma$ and define a map $\hu
  \colon \D \mathbin{\lrsubscripts{\cup}{1}{p}}
  \Sigma \to M\times \D$ by
\[
\hu|_\D (w) = (u(p), zw), \quad 
\hu|_{\Sigma}(y) =  (u(y), z). 
\]
A new interior marking is taken to be $0\in \D$.
If $p$ corresponds to an interior point $\theta
\in (0,\pi)$ of the exceptional locus $[0,\pi]$ of
$\hSigma \to \Sigma$, we attach a disc $\D$ to
$\Sigma$ by identifying $1\in \partial\D$ with
$x\in \partial \Sigma$ and define a map $\hu\colon
\D \mathbin{\lrsubscripts{\cup}{1}{x}} \Sigma \to
M\times \D$ by
\[
\hu|_{\D}(w) = (u(x), e^{-2i\theta}z w), \quad 
\hu|_{\Sigma}(y) = (u(y), e^{-2i\theta} z).  
\]
We put a new boundary marking at $e^{2i\theta} \in
\D$ and a new interior marking at $0\in \D$.  If
$p$ is a boundary point of the exceptional locus
$[0,\pi]$, say, $0\in [0,\pi]$, we consider the
domain $\D^{(1)}
\mathbin{\lrsubscripts{\cup}{1}{-1}} \D^{(2)}
\mathbin{\lrsubscripts{\cup}{1}{x}} \Sigma$
(subscripts signify how to identify boundary
points) with a boundary marking $i\in \D^{(2)}$
and an interior marking $0\in \D^{(1)}$ and define
a map $\hu \colon \D^{(1)} \cup \D^{(2)} \cup
\Sigma \to M\times \D$ by:
\[
\hu|_{\D^{(1)}} (w)= (u(x), zw), \quad 
\hu|_{\D^{(2)}} (w) = (u(x), z), \quad 
\hu|_{\Sigma}(y) = (u(y), z).  
\]
When $p$ corresponds to $\pi\in [0,\pi]$, we take
$-i\in \D^{(2)}$ in place of $i\in\D^{(2)}$ as a
boundary marking.  One can see that the above
construction defines a homeomorphism $\hSigma
\cong \frf^{-1}([u,\Sigma,x],z)$.  The 
last statement is obvious.
\end{proof} 

\begin{figure}[htbp]
\begin{center} 
\begin{picture}(300,100)  
\put(-10,60){\ellipse{30}{80}} 
\put(35,20){\ellipse{83}{20}}
\put(-10,60){\makebox(0,0){\circle*{3}}}
\put(35,10){\makebox(0,0){\circle*{3}}}
\put(32,3){\makebox(0,0){$b$}}
\put(-5,60){\makebox(0,0){$i$}}
\put(35,20){\makebox(0,0){$\alpha$}}

\put(105,60){\ellipse{30}{80}} 
\put(150,20){\ellipse{83}{20}} 
\put(105,60){\makebox(0,0){\circle*{3}}}
\put(110,60){\makebox(0,0){$i$}}
\put(150,20){\makebox(0,0){$\alpha$}} 
\put(90,60){\makebox(0,0){\circle*{3}}}
\put(84,60){\makebox(0,0){$b$}}

\put(220,70){\ellipse{30}{60}}
\put(230,27){\shade\circle{30}}
\put(284,20){\ellipse{80}{20}}
\put(284,20){\makebox(0,0){$\alpha$}} 
\put(219,16){\makebox(0,0){\circle*{3}}}
\put(212,14){\makebox(0,0){$b$}}
\put(220,70){\makebox(0,0){\circle*{3}}}
\put(225,70){\makebox(0,0){$i$}}
\end{picture} 
\end{center} 
\caption{Three types of boundary points of
  $\cM_{1,1}^{\rm rel}(\halpha)$.  The horizontal
  direction is $M$ and the vertical direction is
  $\D$. The boundary/interior markings are denoted
  by $b$ and $i$ respectively. The shaded disc is
  a component where the map is constant.}
\label{fig:boundarypoints} 
\end{figure}
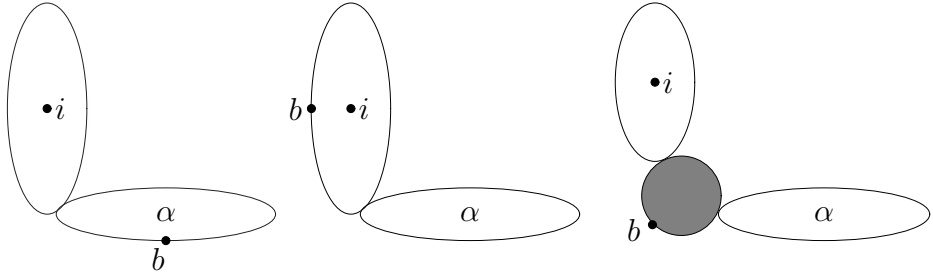 

From the previous lemma and its proof, we have:
\begin{corollary} 
\label{cor:boundarymapstoL} 
Suppose $\partial \cM_1(\alpha) =\emptyset$.  The
boundary $\partial \cM_{1,1}^{\rm rel}(\halpha)$
maps to $L$ under the interior evaluation map
$\ev^{\rm (i)} \colon \cM_{1,1}^{\rm
  rel}(\halpha)\to M$.
\end{corollary} 

\begin{corollary} 
\label{cor:boundary-fibreproduct-empty} 
Suppose $\partial \cM_1(\alpha) = \emptyset$ and
$\ev(\cM_{\rm S}(\sigma)) \cap L = \emptyset$.
Then the fibre product $\cM_{\rm S}(\sigma)
\times_M \cM_{1,1}^{\rm rel}(\halpha)$ has no
boundary.  In particular, the virtual fundamental
cycle $\ev_*[\cM_{\rm S}(\sigma) \times_M
\cM_{1,1}^{\rm rel}(\halpha)]^{\rm vir}$ is
well-defined (see \cite[Lemma A.1.32]{FuOh-La09}).
\end{corollary} 

%

We proceed to calculate the cycle $\ev_*[\cM_{\rm
  S}(\sigma) \times_M \cM_{1,1}^{\rm
  rel}(\halpha)] ^{\rm vir}$ under the assumption
of Corollary
\ref{cor:boundary-fibreproduct-empty}.  By
Corollary \ref{cor:boundarymapstoL}, taking a
sufficiently small perturbation, we get a virtual
fundamental \emph{chain}
\[
\cP_\alpha := 
(\ev^{\rm (i)} \times \ev^{\rm (b)})_*
[\cM_{1,1}^{\rm rel}(\halpha)]^{\rm vir}  
\]
whose boundary lies in $\nu(L)\times L\times S^1$,
where $\nu(L)\subset M$ is an arbitrarily small
tubular neighbourhood of $L$.  In other words,
$\cP_\alpha$ defines a \emph{relative} homology
class of the pair $(M\times L\times S^1, \nu
(L)\times L\times S^1)$ whose dimension is
$n+1+\mu(\alpha)$ (where $n = \dim_\C M$).  On the
other hand, since $\ev(\cM_{\rm S}(\sigma)) \cap L
= \emptyset$, taking a sufficiently small
perturbation again, we obtain a virtual cycle
$\ev_*[\cM_{\rm S}(\sigma)]^{\rm vir}$ in
$M\setminus \nu(L)$.  By Poincar\'{e}-Lefschetz
duality this defines a class
\begin{equation} 
\label{eq:hcS}
\hcS_\sigma := \PD \left( \ev_*[\cM_{\rm S}(\sigma)]^{\rm vir} \right) 
\in H^{\mu(\alpha)-\mu(\beta)}(M, L).
\end{equation} 
Here we used the virtual dimension formula
\eqref{eq:vdim-Seidel} and
\eqref{eq:Maslov+vChern}.  (Note that we put
``hat" to distinguish $\hcS_\sigma \in H^*(M,L)$
from the element $\cS_\sigma\in H^*(M)$ appearing
in \eqref{eq:cS}.)  The virtual cycle of the fibre
product can be evaluated as the pairing of the two
classes:
\begin{equation} 
\label{eq:virtualclass-product} 
\ev_*[\cM_{\rm S}(\sigma) 
\times_M \cM_{1,1}^{\rm rel}(\halpha)]^{\rm vir} 
= \langle \hcS_\sigma, \cP_\alpha \rangle  
\end{equation} 
where $\lan \cdot, \cdot \ran$ is 
the canonical pairing between relative cohomology 
and homology (with K\"{u}nneth decomposition): 
\begin{align*} 
& H^{\mu(\alpha)-\mu(\beta)}(M, \nu(L)) \otimes 
H_{n+1+\mu(\alpha)}(M\times L\times S^1, \nu(L)
\times L\times S^1) \\  
& \quad \longrightarrow 
H^{\mu(\alpha)-\mu(\beta)}(M, \nu(L)) 
\otimes 
H_{\mu(\alpha)-\mu(\beta)}(M,\nu(L)) \otimes 
H_{n+1+\mu(\beta)}(L\times S^1) \\ 
& \quad \longrightarrow H_{n+1+\mu(\beta)}(L\times S^1). 
\end{align*} 
%
We now compute the class \(\cP_\alpha\) in
terms of the class \([\cM_1(\alpha)^{\rm vir}]\). For this,
consider the diagram 
\begin{equation}\label{eq:rel}
\begin{CD} 
\cM_{1,1}^{\rm rel}(\halpha) @>{\frf}>> \cM_1(\alpha) \times S^1 
@>{\ev \times \id}>> L\times S^1 \\ 
@V{\ev^{\rm (i)}}VV @. \\ 
M @. 
\end{CD} 
\end{equation}
where $\frf$ is given in \eqref{eq:forget-udisc}.
The composition of the horizontal arrows is the
boundary evaluation map $\ev^{\rm (b)}$.  As we
saw in Lemma \ref{lem:universaldisc}, $(\frf,
\ev^{\rm (i)})$ can be viewed as a universal
family of bordered stable maps of class $\alpha$.

\begin{lemma} 
Suppose $\partial \cM_1(\alpha)= \emptyset$ 
and $\alpha \neq 0$. 
Then the relative homology class $\cP_\alpha$ is given by 
\begin{equation} 
\label{eq:Palpha} 
\cP_\alpha = \alpha \otimes \left( 
\ev_*[\cM_1(\alpha)]^{\rm vir} \times [S^1] \right) 
\end{equation} 
in $H_*(M\times L\times S^1, \nu(L)\times L\times S^1) 
\cong H_*(M,L)\otimes H_*(L\times S^1)$. 
\end{lemma}

Notice that if one ignores the technical details on the
construction of virtual chains, as well as the
expected functoriality relating the respective chains on
these spaces, the lemma follows directly from Diagram
\eqref{eq:rel}.

\begin{proof} 
We briefly overview the steps towards the proof of this result. 
We first recall the definition of Kuranishi structure and 
the construction of Kuranishi neighbourhoods for $\cM_1(\alpha)$. 
Roughly speaking, a Kuranishi neighbourhood $V$ of 
$\frr_0\in \cM_1(\alpha)$ consists of smooth maps 
from stable discs to $(M,L)$ which are sufficiently ``close'' to $\frr_0$.  
It is equipped with an obstruction bundle $E \to V$ 
satisfying a transversality condition with respect to the Cauchy-Riemann operator, 
together with a section $s$ of $E$ such that a neighbourhood of $\frr_0$ 
in $\cM_1(\alpha)$ is homeomorphic to $s^{-1}(0)\subset V$. 
As a next step, we describe how a Kuranishi neighbourhood 
$V$ for $\cM_1(\alpha)$ induces a Kuranishi neighbourhood 
$\hV$ for $\cM_{1,1}^{\rm rel}(\halpha)$. 
Following exactly the method of Lemma \ref{lem:universaldisc}, 
for a given smooth map $u \colon \Sigma \to M$ in 
the Kuranishi neighbourhood $V$ \emph{and} 
a point of $\hSigma\times S^1$, 
we can canonically construct a smooth map 
$\hu \colon \tSigma \to M\times \P^1$ 
which is holomorphic in the $\P^1$-direction, 
where $\tSigma$ is obtained from $\Sigma$ 
by possibly adding disc or sphere bubbles. 
This constructs a Kuranishi neighbourhood $\hV$ 
for $\cM_{1,1}^{\rm rel}(\halpha)$ which is a fibration 
over $V\times S^1$ with fibres stable discs.   
The key point is that we can take an obstruction bundle over $\hV$ 
to be the \emph{pull-back} of the obstruction bundle $E$ over $V$.  
This allows us to choose the virtual chain of $\cM_{1,1}^{\rm rel}(\halpha)$ 
to be ``fiber bundle'' over a virtual cycle of $\cM_1(\alpha) 
\times S^1$ with fibre the corresponding stable discs, and 
the conclusion of the lemma follows.

\emph{Kuranishi structure on $\cM_1(\alpha)$.}
We refer the reader to \cite[\S 7.1]{FuOh-La09}, \cite[Part 3,
4]{FOOO-technical12} for a detailed description.  
Recall \cite[Definition A1.1]{FuOh-La09} that a
\emph{Kuranishi neighbourhood} of a point $\frr_0
\in \cM_1(\alpha)$ is a tuple $(V,E,\Gamma,\psi,
s)$ where
\begin{itemize} 
\item $V$ is a finite dimensional manifold
  (possibly with boundary and corner);
\item $E$ is a finite dimensional real vector space; 
\item $\Gamma$ is a finite group; it acts on $V$ smoothly 
and effectively and on $E$ linearly; 
\item $s$ is a smooth $\Gamma$-equivariant map 
$V \to E$; 
\item $\psi$ is a homeomorphism between $s^{-1}(0)/\Gamma$ 
and an open neighbourhood of $\frr_0$ in $\cM_1(\alpha)$. 
\end{itemize} 
Every point of $\cM_1(\alpha)$ is equipped with a certain 
Kuranishi neighbourhood, and these Kuranishi neighbourhoods 
are related by certain co-ordinate changes 
\cite[Definition A.1.3]{FuOh-La09} 
and $\cM_1(\alpha)$ becomes a Kuranishi space (a space
endowed with a Kuranishi structure) 
\cite[Definition A.1.5, Proposition 7.1.1]{FuOh-La09}.
The Kuranishi neighbourhoods are cut-off from maps 
satisfying the Cauchy-Riemann equation modulus a finite
dimensional obstruction. Such construction depends on
several parameters, as we now describe in the case at hand.


Let
$(u_0\colon \Sigma_0 \to M, x_0 \in \partial
\Sigma_0)$ be a marked bordered stable map
representing $\frr_0\in \cM_1(\alpha)$.  
The finite group $\Gamma$ is given by the set of
holomorphic automorphisms $\varphi \colon \Sigma_0
\to \Sigma_0$ such that $u_0 \circ \varphi = u_0$
and $\varphi(x_0) = x_0$. Since $(\Sigma_0,x_0)$ has
only one marking, it is unstable if we forget the
map $u_0$.  We add 
interior markings $w_{0,1}, \dots, w_{0,l}$ in $\Sigma_0$ so that
$(\Sigma_0, x_0, \{w_{0,1},\dots,w_{0,l}\})$ is now stable.  
We require that the set $\{w_{0,1},\dots,w_{0,l}\}$ is preserved
by the $\Gamma$-action \cite[Definition 17.5]{FOOO-technical12}, 
so that \(\Gamma\) permutes the indices. 
We can therefore regard \(\Gamma\) as a subgroup 
of the symmetric group $\frS_l$.  
We take real codimension 2 submanifolds $Q_1,\dots,
Q_l$ of $M$ such that $u_0$ intersects
$Q_i$ transversely at $w_{0,i}$ (so $u_0$ is
necessarily an immersion at $w_{0,i}$); moreover
we require that $Q_i = Q_{\sigma(i)}$ for every
permutation $\sigma\in \Gamma \subset \frS_l$.
Let $\cM_{1,l}$ denote the moduli space of
genus-zero stable bordered Riemann surfaces with
one boundary and $l$ interior markings and let
$\fra_0\in \cM_{1,l}$ be the point represented by
$(\Sigma_0,x_0,\{w_{0,1},\dots,w_{0,l}\})$.  The
group $\Gamma$ acts on $\cM_{1,l}$ by permutation
of the $l$ interior markings and $\fra_0$ is fixed by
$\Gamma$.  Let $N\subset \cM_{1,l}$ be a
$\Gamma$-invariant small open
neighbourhood of $\fra_0$. Fukaya-Oh-Ohta-Ono
\cite[Section 16]{FOOO-technical12} constructed $N$ in two
steps. First they considered a subset
$\mathfrak{V}\subset \cM_{1,l}$ consisting of
deformations of $\fra_0$ having the same combinatorial type,
i.e. same associated dual graph as $\fra_0$. Then, they introduced 
smoothing (or gluing) near the nodes with parameters $T\in \R$ and 
$\theta\in S^1$. 
This constructs the neighbourhood $N$ of $\mathfrak{V}$. 
This construction yields a \(\Gamma\)-equivariant 
\emph{tautological} family $\cR \to N$ 
\cite[Lemma 16.9]{FOOO-technical12} 
of stable bordered Riemann surfaces, where the fibre at
\(\frb\in N\) corresponds to its underlying surface. 
Note that when $\gamma \cdot \frb = \frb'$ for $\frb,\frb' \in N$ 
and $\gamma \in \Gamma$, there is a canonical isomorphism 
between the underlying surfaces of $\frb$ and $\frb'$ 
which induces the permutation $\gamma$ of the interior markings; 
this defines the $\Gamma$-action on $\cR$. 
We take a $\Gamma$-invariant
closed subset $\cK \subset \cR$ such that the
fibre $K_0=\Sigma_0\cap \cK$ at $\fra_0$ is the
complement in $\Sigma_0$ of small neighbourhoods
of the nodes in $\Sigma_0$, and that the family $\cK
\to N$ is $C^\infty$-trivial.  
We choose a $\Gamma$-equivariant $C^\infty$-trivialization
$\cK \cong K_0 \times N$ which preserves the markings, 
i.e.~the section of the $i$th interior marking in $\cK$ 
corresponds to 
$\{w_{0,i}\} \times N \subset K_0\times N$.  
$\cK$ is called the \emph{core} and its complement is called 
the \emph{neck region} (
for further details see \cite[Definitions 16.2, 16.4, 16.6, 16.7]
{FOOO-technical12}.)
%

For a bordered Riemann surface $\Sigma$
appearing as a fibre of $\cR\to N$, the core
$K=\Sigma \cap \cK$ is identified with $K_0$ by
the given trivialization $\cK \cong K_0\times N$,
and thus $u_0$ induces a map $u_0 \colon K \to M$.
We consider an infinite dimensional space $\cU$
consisting of tuples $(u,\Sigma, x,
\{w_1,\dots,w_l\})$, where $(\Sigma, x,
\{w_1,\dots,w_l\})$ represents a point of $N$ and
$u \colon (\Sigma, \partial \Sigma) \to (M,L)$ is
a smooth map of degree $\alpha$ which is
``sufficiently close'' to $u_0$ in the sense that
(see \cite[Definitions 17.12,
18.10]{FOOO-technical12}) for an \(\varepsilon>0\)
\begin{itemize} 
\item $u$ is \(\varepsilon\)-close to $u_0$ in the
  $C^{10}$-topology on the core
  $K = \Sigma\cap \cK$;
\item $u$ is holomorphic on the neck region
  $\Sigma \setminus K$;
\item the diameter of the image of each connected component 
of the neck region under $u$ is smaller than \(\varepsilon\). 
\end{itemize} 
The group $\Gamma$ acts on $\cU$ by permutation of
interior marked points.  Next we choose an
obstruction bundle $\E$ over $\cU$ as follows (see
\cite[Definitions 17.7, 17.15]{FOOO-technical12}).
Take a $\Gamma$-equivariant smooth family of
finite dimensional subspaces $\E_\fra$
\[
\E_\fra \subset C^\infty_c ( \Int(K), u_0^*TM
\otimes \Lambda^{0,1})
\] 
parametrized by $\fra = (\Sigma, x,
\{w_1,\dots,w_l\})\in N$, where $K = \Sigma \cap
\cK$ is the core of $\Sigma$ and $\Lambda^{0,1}$
is the bundle of $(0,1)$-forms on $\Sigma$.  Then
we extend this family to the whole $\cU$ via
parallel transport, i.e.~for each point $\frr
=(u,\Sigma, x, \{w_1,\dots,w_l\})\in \cU$ over
$\fra = (\Sigma, x, \{w_1,\dots,w_l\}) \in N$, we
define
\[
\E_{\frr} \subset C^\infty_c(\Int(K), u^* TM
\otimes \Lambda^{0,1})
\]
as the parallel transport of $\E_{\fra}$ along
geodesics joining $u(y)$ and $u_0(y)$, for \(y\in \Int(K)\).
Here we
use a connection on $TM$ such that $TL$ is
preserved by parallel translation \cite[\S
11]{FOOO-technical12}.  By construction, the
bundle $\E \to \cU$ is $\Gamma$-equivariant.  The
\emph{Kuranishi neighbourhood} $V \subset \cU$ is cut
out by the equations:
\begin{align} 
\label{eq:Knbhd} 
\begin{split}
& \ol\partial u \equiv 0 \mod \E_{\frr} \\ 
& u(w_i) \in Q_i \quad i=1,\dots,l 
\end{split}  
\end{align} 
for $\frr = (u,\Sigma, x,\{w_1,\dots,w_l\}) \in
\cU$.  We need to choose $\E$ so that the
equations \eqref{eq:Knbhd} are transversal (see
below).  The $\Gamma$-action on $\cU$ preserves
$V$ and thus the obstruction bundle restricts to a
$\Gamma$-equivariant vector bundle $E = \E|_V$
over $V$.  The Cauchy-Riemann operator
$\ol{\partial}$ induces a section $s$ of $E \to V$
and $s^{-1}(0)/\Gamma$ gives a neighbourhood of
$\frr_0 \in \cM_1(\alpha)$.

The required transversality for \eqref{eq:Knbhd}
is stated as follows (see \cite[Lemmata 18.16,
20.7]{FOOO-technical12}).  For a smooth map $u
\colon (\Sigma, \partial \Sigma) \to (M,L)$, let
$L^2_{m,\delta}(\Sigma,\partial \Sigma; u^* TM,
u^*TL)$ denote a weighted Sobolev space, 
for $m$ sufficiently large and  \(\delta>0\), 
and consisting of $L^2_{m,\, {\rm loc}}$-sections of
$u^* TM$ which take values in $u^* TL$ along the
boundary $\partial \Sigma$, see \cite[Definitions
10.1, 19.8]{FOOO-technical12}.
Let $L^2_{m,\delta}(\Sigma, u^*TM\otimes
\Lambda^{0,1})$ denote a similar weighted Sobolev
space of sections of $u^* TM \otimes
\Lambda^{0,1}$ (see \cite[Definition
19.9]{FOOO-technical12}).  Let 
\[
D_{\frr} \ol{\partial} \colon L^2_{m+1,\delta}
(\Sigma, \partial \Sigma; u^* TM, u^* TL) \to
L^2_{m,\delta} (\Sigma, u^*TM \otimes
\Lambda^{0,1})
\]
denote the linearization of
$\ol{\partial}$ at $\frr = (u, \Sigma, x,
\{w_1,\dots, w_l\}) \in \cU$. 
We require
that $\Im (D_{\frr} \ol{\partial})$ and
$\E_{\frr}$ span $L^2_{m,\delta}(\Sigma, u^*
TM\otimes \Lambda^{0,1})$ for each $\frr \in
\cU$. (This is called ``Fredholm regularity".)
Let $\cM\subset \cU$ denote the subspace cut out
only by the first equation of \eqref{eq:Knbhd}.
Let $\ev_{\rm ad} \colon \cM \to M^{l}$ be the 
evaluation map at the $l$ additional markings. 
We also require $\ev_{\rm ad}$ to be transversal to 
$\prod_{i=1}^l Q_i \subset M^l$. 
Then $V = \ev_{\rm ad}^{-1}( \prod_{i=1}^l Q_i )$ is the
desired neighbourhood.   

\emph{Induced Kuranishi structure on \(\cM_{1,1}^{\rm
rel}(\halpha)\).} Recall that \(\frf\) is the forgetful morphism 
\(\cM_{1,1}^{\rm rel}(\halpha) \to \cM_1(\alpha) \times S^1\)   
 as in \eqref{eq:forget-udisc}.
We now construct a Kuranishi neighbourhood of
$\frf^{-1}(\frr_0 \times S^1) \subset
\cM_{1,1}^{\rm rel}(\halpha)$ from the Kuranishi
neighbourhood $(V,E,\Gamma,\psi,s)$ of $\frr_0\in
\cM_1(\alpha)$ above. We have
$\frf^{-1}(\frr_0 \times S^1) \cong \hSigma_0
\times S^1$ by Lemma \ref{lem:universaldisc},
where $\hSigma_0$ is the oriented real blow-up of
$\Sigma_0$ at $x_0$.  We perform this oriented
real blow-up in families.  The family $\cR \to N$
is equipped with a section $x\colon N \to \cR$
corresponding to the boundary marked point.  Let
$\hcR$ denote the oriented real-blow up along
 $x$.  The proof of Lemma
\ref{lem:universaldisc} shows that a point $p\in
\hcR$ parametrizes a marked stable bordered
Riemann surface\footnote {By abuse of notation, we
  denote by $p$ a point of $\hcR$ and at the same
  time a new interior marking on $\tSigma$.}
$(\tSigma, x, p, \{w_1,\dots,w_l\})$ with a new
interior marking $p$ (see also \cite[Lemma
7.1.45]{FuOh-La09}).  More precisely, letting
$p\in \hcR$ be on the blow-up of a fibre
$\Sigma\subset \cR$: if $p$ is neither a node nor
a boundary point, $\tSigma = \Sigma$; if $p$ is an
interior node, $\tSigma$ is obtained from $\Sigma$
by adding a sphere bubble at the node; if $p$ is a
boundary point, $\tSigma$ is obtained from
$\Sigma$ by adding at most two disc bubbles (see
Figure \ref{fig:boundarypoints}).  
It is possible that a new interior marking $p$ on $\tSigma$ 
coincides with one of the $w_i$'s.  
Let $\frR \to \hcR$ denote the corresponding family 
of marked stable bordered Riemann surfaces.
The $\Gamma$-action on $\hcR$ naturally lifts to 
that on the tautological family $\frR\to \hcR$. 
Let $p\in \hcR$ be a point on the blow-up $\hSigma$ 
of a fiber $\Sigma\subset \cR$ of $\cR \to N$ and 
let $\tSigma$ be the fibre of $\frR \to \hcR$ at $p\in \hcR$ as above. 
Then the core $K = \cK \cap \Sigma$ of $\Sigma$ induces 
a compact subset $\tK$ of $\tSigma$ which maps isomorphically 
onto $K$ under the natural map $\tSigma \to \Sigma$. 
(The set $\tK$ is like the strict transform of $K$.)
The union of these subsets $\tK$ gives a 
$\Gamma$-invariant subset $\frK \subset \frR$  
equipped with a $\Gamma$-equivariant 
$C^\infty$-trivialization $\frK \cong K_0\times \hcR$. 
Notice that $\Int(\frK)$ is disjoint from the components
contracted under $\tSigma \to \Sigma$. 

Let $\frU$ be the space of tuples
$(\hu, \tSigma, x, p, \{w_1,\dots,w_l\})$ where
$(\tSigma, x, p, \{w_1,\dots,w_l\})$ is a marked
bordered Riemann surface corresponding to a point
of $\hcR$ (i.e.~ arises as a fibre of $\frR \to
\hcR$) and $\hu \colon (\tSigma,\partial \tSigma)
\to (M\times \P^1, L\times S^1)$ is a smooth map
of class $\halpha$ which satisfies the following
conditions:
\begin{itemize} 
\item $\pi_M \circ \hu$ is $C^{10}$-close to $u_0$
  on $\tK = \tSigma \cap \frK$, where
  $\pi_M \colon M\times \P^1 \to M$ is the
  projection (since $\tK$ is identified with $K_0$
  via the given trivialization $\frK \cong
  K_0\times \hcR$, $u_0$ defines a map $u_0 \colon
  \tK \to M$);
\item $\hu$ is holomorphic on 
  $\tSigma \setminus \tK$;
\item the diameter of the image of each connected
  component of $\tSigma \setminus \tK$ under $\pi_M\circ
  \hu$ is small.
\end{itemize} 
The obstruction bundle $\E\to \cU$  induces an obstruction
bundle  $\bE \to \frU$ as follows.  Take an
element $\frs = (\hu, \tSigma,
x,p,\{w_1,\dots,w_l\}) \in \frU$ and let $\fra =
(\Sigma, x, \{w_1,\dots, w_l\}) \in N$ denote the
marked Riemann surface given by forgetting $p$ and
collapsing unstable components of the source.
Define the obstruction space at $\frs \in \frU$
\[
\bE_{\frs} \subset C^\infty_c(\Int(\tK),
(\pi_M\circ \hu)^* TM \otimes \Lambda^{0,1})
\subset C^{\infty}_c ( \Int(\tK), \hu^* T(M\times
\P^1) \otimes \Lambda^{0,1})
\]
(with $\tK = \tSigma \cap \frK$) to be the
parallel transport of $\E_\fra\subset
C^{\infty}_c(\Int(\tK), u_0^*TM \otimes
\Lambda^{0,1})$ along geodesics joining $u_0(y)$
and $(\pi_M \circ \hu)(y)$.  Let $C \subset
\tSigma$ be the union of the contracted components of $\tSigma
\to\Sigma$.  Because $\pi_M \circ \hu|_{\tK}$ is
sufficiently close to $u_0$, $\pi_M\circ \hu|_{\tSigma \setminus \tK}$ 
is holomorphic and $C \subset \tSigma \setminus \Int(\tK)$, 
by choosing a smaller neck region from the 
beginning if necessary (see ``extending the core"
\cite[Definition 17.21]{FOOO-technical12}), we may
assume that $\pi_M \circ \hu$ is constant on $C$
(since the symplectic area of $(\pi_M\circ \hu)(\tSigma \setminus \tK)$ 
has to be small).  Hence $\pi_M \circ \hu$ induces a
map $u \colon (\Sigma,\partial \Sigma) \to (M,L)$
belonging to $\cU$. Therefore we have a projection
$\frU \to \cU$ and $\bE$ is identified with the
pull-back of $\E$.  The group $\Gamma$ acts on
$\frU$ and $\bE$ and $\bE \to \frU$ is
$\Gamma$-equivariant.  The Kuranishi neighbourhood
$\hV$ for $\cM_{1,1}^{\rm rel}(\halpha)$ is cut
out from $\frU$ by the equations:
\begin{align} 
\label{eq:Knbhd-upstairs} 
\begin{split} 
& \ol{\partial} \hu \equiv 0 \mod \bE_\frs \\
& \hu(p) \in M\times \{0\} \\ 
& \hu(w_i) \in Q_i \times \P^1 \quad i=1,\dots, l
\end{split} 
\end{align} 
for $\frs = (\hu, \tSigma, x, p,
\{w_1,\dots,w_l\}) \in \frU$.  The second equation
of \eqref{eq:Knbhd-upstairs} corresponds to the
fibre product presentation \eqref{eq:M11rel} of
$\cM_{1,1}^{\rm rel}(\halpha)$.  

Let $\hcM\subset
\frU$ denote the subspace cut out by the first and
the second equations of \eqref{eq:Knbhd-upstairs}.
Consider the map $\frU \to \cU \times S^1$, where
the first factor is the projection we discussed
and the second factor is the evaluation map at the
boundary marking $x$ followed by the projection
$L\times S^1 \to S^1$.  We claim that $\hcM$ is a
tautological family of (blown-up) Riemann surfaces
over $\cM\times S^1$ under the map $\hcM \subset
\frU \to \cU\times S^1$.  (Recall that $\cM\subset
\cU$ is cut out by the first equation of
\eqref{eq:Knbhd}.)  More precisely, it is
identified with the restriction to $\cM\times S^1$
of the family $\pr^*\hcR \to \cU\times S^1$ where
$\pr \colon \cU\times S^1 \to N$ is the
natural projection.  By the choice of $\bE$, each
element $(\hu, \tSigma, x, p, \{w_1,\dots,w_l\})$
of $\hcM$ is holomorphic in the $\P^1$-factor and
its image $(u,\Sigma, x, \{w_1,\dots,w_l\})$ in
$\cU$ belongs to $\cM$. By the same argument as in
the proof of Lemma \ref{lem:universaldisc}, it
follows that $\hu$ is uniquely reconstructed from
$u\colon \Sigma \to M$, $p\in \hSigma$ and
$(\pi_{\P^1}\circ \hu)(x)\in S^1$.  This proves
the claim.  Cutting down the moduli space $\hcM$
by the third equation of
\eqref{eq:Knbhd-upstairs}, we obtain $\hV$ as a
tautological family of (blown-up) Riemann surfaces
over $V\times S^1$, with $V$ the Kuranishi
neighbourhood of $\frr_0 \in \cM_1(\alpha)$.  The
obstruction bundle $\hE = \bE|_{\hV}$ and its
section $\hs :=\ol{\partial}$ are the pull-backs
of $E\to V$ and $s = \ol{\partial}$
respectively. These data $(\hV, \hE,\hs)$ are
$\Gamma$-equivariant and give a Kuranishi
neighbourhood $(\hV,\hE, \Gamma, \hpsi, \hs)$ of
$\frf^{-1}(\frr_0\times S^1) \subset
\cM_{1,1}^{\rm rel}(\halpha)$.

We need to check the transversality of \eqref{eq:Knbhd-upstairs}.  
We first show the transversality of the $\ol{\partial}$-equation. 
Let $\hu \colon \tSigma \to M$ be a map in $\frU$ 
satisfying the first equation of \eqref{eq:Knbhd-upstairs}. 
Let $v := \pi_{\P^1} \circ \hu \colon (\tSigma,\partial \tSigma) 
\to (\P^1,S^1)$ be the vertical component of $\hu$, 
which is a holomorphic map of degree one. 
Let $\tu = \pi_M \circ \hu \colon (\tSigma,\partial \tSigma) 
\to (M,L)$ be the horizontal component of $\hu$. 
The image $u \colon (\Sigma,\partial\Sigma) \to (M,L)$ 
of $\hu$ in $\cU$ is obtained from $\tu$ by collapsing 
some components of $\tSigma$ on which $\tu$ is 
constant. 
It suffices to show that 
\begin{itemize} 
\item $v$ is Fredholm regular, i.e.~$D_v \ol{\partial}$ is surjective; and  
\item $\tu$ is Fredholm regular for $\E_{u}$, 
i.e.~$\Im (D_{\tu}\ol{\partial}) + \E_{u}  
= L_{m,\delta}^2(\tSigma, \tu^*(TM)\otimes \Lambda^{0,1})$. 
\end{itemize} 
The Fredholm regularity for $v$ can be rephrased as the vanishing 
of the sheaf cohomology (see \cite[\S 3.4]{KaLi-06}, 
\cite[\S 6]{ChOh-Fl06}):
\[
H^1(\tSigma, (v^* T\P^1, v^* TS^1))=0
\]
where $(v^* T\P^1, v^* TS^1)$ denotes the sheaf of
holomorphic sections of $v^* T\P^1$ which take
values in $v^* TS^1$ on $\partial \tSigma$. 
Let $\tSigma = \bigcup_i \Sigma_i$ 
be the decomposition into irreducible components 
and write $v_i = v|_{\Sigma_i}$. 
Then we have the following standard normalization 
sequence for sheaves on $\tSigma$: 
\[
0 \to (v^*T\P^1,v^*TS^1) \to 
\bigoplus_{i} (v_i^*T\P^1, v_i^*TS^1) 
\to \bigoplus_{x} T_{v(x)}\P^1
\oplus \bigoplus_{y} T_{v(y)}S^1 
\to 0 
\]
where $x$ ranges over interior nodes of $\tSigma$ 
and $y$ ranges over boundary nodes of $\tSigma$. 
Since we have $H^1(v_i^*T\P^1,v_i^*TS^1) = 0$ for 
each component $\Sigma_i$ by \cite[Lemma 6.4]{ChOh-Fl06}, 
it suffices to show that the map 
$\bigoplus_i H^0(v_i^*T\P^1,v_i^*TS^1) 
\to \bigoplus_x T_{v(x)}\P^1 \oplus \bigoplus_y T_{v(x)}S^1$ 
is surjective: this follows easily by induction on the number 
of components (removing degree-zero tails one by one). 
The Fredholm regularity of $\tu$ with respect to $\E$ follows from 
the assumed regularity for $u$ with respect to $\E$. 
For example, consider the case where $\tSigma = \Sigma \cup \D$. 
The obstruction space $\E_u$ is supported on the core $K \subset \Sigma$ 
and $\tu$ is constant on $\D$. Write $x = \tu(\D) \in M$. 
Let $(\xi_1,\xi_2)$ be an element of 
\[
L^2_{m,\delta}(\tSigma, \tu^*TM \otimes \Lambda^{0,1}) 
= L^2_{m,\delta}(\Sigma, u^*TM\otimes \Lambda^{0,1}) 
\oplus L^2_{m,\delta}(\D, T_x M\otimes \Lambda^{0,1}). 
\]
The assumed regularity implies that 
there exists  $\epsilon \in \E_u$ and 
$\nu_1 \in L^2_{m+1,\delta}(\Sigma,\partial \Sigma; u^*TM, u^*TL)$ 
such that $\xi_1 = (D_u\ol{\partial})\nu_1 + \epsilon$. 
The vanishing $H^1(T_xM, T_xL)=0$ of the sheaf 
cohomology on $\D$ implies that there exists 
$\nu_2 \in L^2_{m+1,\delta}(\D,\partial \D; T_xM, T_xL)$ 
such that $\ol{\partial} \nu_2 = \xi_2$. 
By adding a constant element in $T_xL$ to $\nu_2$, we may 
assume that $\nu_1$ and $\nu_2$ agree on the boundary 
node $\D \cap \Sigma$, and then we have 
$(\xi_1,\xi_2) = (D_{\tu}\ol{\partial}) (\nu_1,\nu_2) 
+ \epsilon$. This shows the regularity of $\tu$. 
The other cases are similar.


Let $\cN\subset \frU$ denote the moduli space cut out
only by the first equation of
\eqref{eq:Knbhd-upstairs}.  The holomorphic
automorphism group $\Aut(\D)$ acts on the target
$(M\times \P^1, L \times S^1)$ and also on the
moduli space $\cN$.  The transversality for the
second equation of \eqref{eq:Knbhd-upstairs}
follows from the fact that the $\Aut(\D)$-action
on $\Int(\D)$ is transitive.  The first and the
second equations of \eqref{eq:Knbhd-upstairs}
define the modui space $\hcM$.  The evaluation map
$\ev_{\rm ad} \colon \hcM \to (M\times \P^1)^l$ at
the markings $w_1,\dots,w_l$ is transversal to
$\prod_{i=1}^l (Q_i\times \P^1)$ by the
transversality assumption for the second equation of \eqref{eq:Knbhd}. 
The transversality for \eqref{eq:Knbhd-upstairs} is now proved. 

\emph{Comparison of virtual cycles.} A virtual
chain is defined by multi-valued perturbations
(multisections) of $s$ on Kuranishi neighbourhoods
which are compatible under co-ordinate changes, and it is
independent of the choice of the obstruction bundle
(see \cite[\S A1.1]{FuOh-La09}, \cite[Part
2]{FOOO-technical12}).  By the above construction
of Kuranishi neighbourhoods, and the hypothesis \(\partial
\cM_1(\alpha)= \emptyset\), 
we can define a
virtual cycle $[\cM_{1,1}^{\rm rel}(\halpha)]^{\rm
  vir}$ by pulling back multisections used to
define a virtual cycle $[\cM_1(\alpha)]^{\rm
  vir}$, and this is independent of choices. 
Then $[\cM_{1,1}^{\rm rel}(\halpha)]^{\rm vir}$ becomes a fibre bundle
over $[\cM_1(\alpha)]^{\rm vir}\times S^1$ with
fibre the corresponding stable bordered Riemann
surfaces.  Each fibre is of class $\alpha$ under
the interior evaluation map. The lemma follows.
\end{proof} 

Summarizing the discussion,  we obtain (see
\eqref{eq:alphazero},  
\eqref{eq:virtualclass-product}, \eqref{eq:Palpha}): 
\begin{lemma} 
\label{lem:rhs}
The virtual cycle $\ev_* [\cM_{\rm S}(\sigma)
\times_M \cM_{1,1}^{\rm rel}(\halpha)]^{\rm vir}$ 
is well-defined if one of the following holds: 
\begin{itemize} 
\item[(a)] $\cM_{\rm S}(\sigma)=\emptyset$ or; 

\item[(b)] $\cM_1(\alpha)=\emptyset$ or; 

\item[(c)] $\partial \cM_1(\alpha) =\emptyset$ and 
$\ev(\cM_{\rm S}(\sigma)) \cap L = \emptyset$. 
\end{itemize} 
When one of the above conditions holds, we have:  
\[
\ev_* [\cM_{\rm S}(\sigma) \times_M \cM_{1,1}^{\rm
  rel}(\halpha)]^{\rm vir} = \begin{cases}
  (\cS_{\sigma} \cap [L]) \times [S^1] & \text{if
  } \alpha=0
  \text{ (then (b) holds);} \\
  \langle \hcS_\sigma, \alpha \rangle
  \ev_*[\cM_1(\alpha)]^{\rm vir}
  \times [S^1] & \text{if  (c) holds;} \\
  0 & \text{if $\alpha \neq 0$ and (a) or (b)
    holds.}
\end{cases} 
\]
\end{lemma}

\subsubsection{Conjecture and expected results} 
\label{subsubsec:conjecture}
We now state our conjecture: 
\begin{conjecture}[Degeneration Formula] 
\label{conj:degeneration} 
Let $\beta \in H_2(M,L)$ be such that $\cM_1(\beta)=\emptyset$. 
Assume that every pair $(\sigma, \alpha)\in H_2^{\rm sec}(E) 
\times H_2(M,L)$ with 
$r(\hbeta) = \sigma + \halpha$  satisfies one of the 
three conditions (a), (b), (c)  in Lemma \ref{lem:rhs}. 
Then the degeneration formula \eqref{eq:degenerationformula} 
\[
\varphi_* \ev_*[\cM_1(\hbeta) ]^{\rm vir} =
\sum_{r(\hbeta) = \sigma + \halpha} \ev_*
[\cM_{\rm S} (\sigma)\times_M \cM^{\rm
  rel}_{1,1}(\halpha)]^{\rm vir}
\]
holds. 
This implies, by Lemmata \ref{lem:lhs} and \ref{lem:rhs}, that 
\begin{equation} 
\label{eq:degenerationformula-explicit} 
 \delta_{\beta,0} [L] 
= \sum_{\substack{(\sigma, \alpha) : \, 
r(\hbeta) = \sigma + \halpha, \ \alpha \neq 0\\ 
\text{satisfying (c) of Lemma \ref{lem:rhs}}}} 
 \langle \hcS_\sigma, \alpha \rangle
 \ev_*[\cM_1(\alpha)]^{\rm vir}  
 + \delta_{\partial \beta, \lambda}  
\sum_{\sigma: r(\hbeta) = \sigma + \hzero}
\cS_{\sigma} \cap [L] 
\end{equation} 
holds in $H_{n+\mu(\beta)}(L;\Q)$. 
Here $\cS_\sigma$, $\hcS_\sigma$ are defined in 
\eqref{eq:cS}, \eqref{eq:hcS} and $\lambda \in H_1(L)$ 
is the class of an $S^1$-orbit. 
\end{conjecture} 

Note that the second term in the right-hand side 
of \eqref{eq:degenerationformula-explicit} arises from 
the case $\alpha=0$ 
(recall the discussion around \eqref{eq:alphazero}). 

In practice it is not easy to make all the assumptions here 
to be satisfied and to obtain a non-trivial result 
from \eqref{eq:degenerationformula-explicit}. 
Notice that the both-hand sides of \eqref{eq:degenerationformula-explicit} 
are zero unless $\mu(\beta)\le 0$ for dimensional reason. 
Also the term $\langle \hcS_\sigma, \alpha \rangle$ 
is zero unless $\hcS_\sigma \in H^2(M,L)$, i.e.\ 
$\lan c_1^{\rm vert}(E),\sigma \ran = -1$. 
Hence by \eqref{eq:Maslov+vChern}, 
the first term of the right-hand side 
is the sum over classes $\alpha$ satisfying $
\mu(\alpha) = \mu(\beta) +2$. 
This motivates the following (rather restrictive) assumption: 
\begin{assumption} 
\label{assump:moduli} 
(i) $\cM_1(\beta)$ is empty for all $\beta
\in H_2(M,L) $ with $\mu(\beta) \le 0$.  

(ii) The maximal fixed component $F_{\rm max}\subset M$ of 
the $\C^\times$-action (see \S \ref{subsec:Seidel}) 
is of complex codimension one and the $\C^\times$-weight 
on the normal bundle is $-1$. 

(iii) $c_1(M)$ is semi-positive. 

(iv) $\ev(\cM_{\rm S}(\sigma))$ is disjoint from $L$ 
for all $\sigma\in H_2^{\rm sec}(E)$ such that 
$\lan c_1^{\rm vert}(E), \sigma \ran = -1$.  
\end{assumption} 

We assume Assumption \ref{assump:moduli} 
in the rest of this section. 
Recall from Definition \ref{def:W-general} that 
open Gromov-Witten invariants $n_\alpha$ are defined 
when $\mu(\alpha)=2$ by the assumption (i) 
and so the potential function $W$ of $L$ is also defined. 
The role of the assumptions (ii) and (iii) is as follows. 
The assumption (ii) implies that 
$\lan c_1^{\rm vert}(E),\sigma_0\ran = -1$ 
for a maximal section $\sigma_0$.  
Note that by \eqref{eq:split-sigma0} 
$\cM_{\rm S}(\sigma)$ is empty unless 
$\sigma = \sigma_0+ d$ for some $d\in \NE(M)_\Z$. 
Therefore by (iii), $\cM_{\rm S}(\sigma)$ is empty 
unless 
$\lan c_1^{\rm vert}(E), \sigma \ran \ge -1$. 
This implies that the Seidel element $S$ in Definition 
\ref{def:seidel} is in 
$H^{\le 2}(M;\Q) \otimes \Q[\![\NE(E)_\Z]\!]$. 
\begin{definition} 
\label{def:Seidellift}
Under Assumption \ref{assump:moduli}, we can decompose the Seidel element as  
\[
S = q_0 \tS = q_0(\tS^{(0)} + \tS^{(2)}) 
\]
with $\tS^{(i)} \in H^i (M;\Q) \otimes \Q[\![\NE(M)_\Z]\!]$ 
and $q_0 = q^{\sigma_0}$.   
Furthermore, we can define a lift $\hS^{(2)}$ of $\tS^{(2)}$ 
as follows: 
\[
\hS^{(2)} := \sum_{\sigma:\, \langle c_1^{\rm vert}(E), \sigma \rangle 
= -1} \hcS_\sigma q^{\sigma-\sigma_0}  
\]
where $\hcS_\sigma\in H^2(M,L;\Q)$ (see \eqref{eq:hcS}) 
is well-defined by Assumption \ref{assump:moduli} (iv). 
The lift $\hS^{(2)}$ is an element of 
$H^2(M,L;\Q) \otimes \Q[\![\NE(M)_\Z]\!]$ 
which maps to $\tS^{(2)}$ under the natural map 
$H^2(M,L) \to H^2(M)$. 
\end{definition}

Under Assumption \ref{assump:moduli}, 
the conditions in Conjecture \ref{conj:degeneration} 
are satisfied for all $\beta$ with $\mu(\beta)=0$. 
In fact, if $r(\hbeta) = \sigma + \halpha$, then 
$\mu(\alpha) + 2 \lan c_1^{\rm vert}(E), \sigma\ran= 0$ 
by \eqref{eq:Maslov+vChern}, and thus 
\begin{itemize} 
\item 
if $\mu(\alpha) \le 0$ and $\lan c_1^{\rm vert}(E), \sigma \ran \ge 0$, then 
$\cM_1(\alpha) = \emptyset$ by the assumption (i); 
\item if $\mu(\alpha) \ge 4$ and $\lan c_1^{\rm vert}(E),\sigma \ran \le -2$, 
then
$\cM_{\rm S}(\sigma) = \emptyset$ by the assumptions (ii), (iii);  
\item 
if $\mu(\alpha)=2$ and $\lan c_1^{\rm vert}(E), \sigma \ran =-1$,
then $\cM_1(\alpha)$ has no boundary 
by the assumption (i) and $\ev(\cM_{\rm S}(\sigma)) \cap L = \emptyset$ 
by the assumption (iv).  
\end{itemize} 
Fix a class $\gamma \in H_1(L)$.  We now apply the
formula \eqref{eq:degenerationformula-explicit}
for $\beta$ with $\mu(\beta) =0$ and $\partial
\beta = \gamma+\lambda$.  In this case, 
\eqref{eq:degenerationformula-explicit} yields the
following equality in $H_n(L;\Q) \cong \Q$:
\begin{equation} 
\label{eq:formula-underassump} 
\delta_{\beta,0} = \sum_{\substack{
    (\sigma, \alpha): \, r(\hbeta) = \sigma+ \halpha 
     \\ 
    \mu(\alpha) =2, \ \langle c_1^{\rm vert}(E),
    \sigma \rangle =-1 
  }}
\langle \hcS_\sigma, \alpha \rangle n_{\alpha} 
+ \delta_{\partial \beta, \lambda} 
\sum_{\sigma: \, r(\hbeta) = \sigma + \hzero} 
\cS_\sigma 
\end{equation} 
where $n_\alpha$ is the open Gromov-Witten
invariant defined in Definition
\ref{def:W-general}.  Note that $\cS_\sigma$ in
the second term of the right-hand side lies in
$H^0(L;\Q) \cong \Q$.  
We consider a generating
function in the ``open" Novikov ring
$\Lambda^{\rm op}$ which was introduced before
Definition \ref{def:W-general}. 
We have a (not necessarily injective) homomorphism  from 
the ``closed" Novikov ring $\Lambda$ (see Remark \ref{rem:Seidelrep}) 
to the ``open" Novikov ring $\Lambda^{\rm op}$  
\[
\Lambda \to \Lambda^{\rm op},\quad 
q^d \mapsto z^d.  
\]
Thus $\Lambda^{\rm op}$ is a $\Lambda$-algebra. 
Note that $r(\hbeta) = \sigma +\halpha$ means 
\[
z^{\alpha_0+\beta} = q^{\sigma-\sigma_0} z^{\alpha} 
 \quad 
 \text{in $\Lambda^{\rm op}$} 
\]
by Proposition \ref{prop:retractionmap} 
where $\sigma_0$, $\alpha_0$ are maximal section/disc classes. 
We multiply the both-hand sides of \eqref{eq:formula-underassump} 
by $z^{\alpha_0+\beta} = q^{\sigma-\sigma_0} z^{\alpha}$ 
and sum over all $\beta$ with $\mu(\beta)=2$ and $\partial \beta =
\gamma + \lambda$. About the first term of the right-hand side, 
this summation boils down to the sum over all $(\sigma, \alpha)$ 
with $\lan c_1^{\rm vert}(E),\sigma \ran =-1$, 
$\mu(\alpha)=2$, $\partial \alpha =\gamma$ 
(see \eqref{eq:boundaryrelation}); 
about the second term of the right-hand side (which occurs 
when and only when $\gamma=0$), this boils down 
to the sum over all $\sigma$ with 
$\lan c_1^{\rm vert}(E), \sigma \ran=0$. 
Therefore we have:  

\begin{theorem} 
\label{thm:degeneration} 
Assume that the degeneration formula (Conjecture \ref{conj:degeneration}) 
and Assumption \ref{assump:moduli} hold for $(M,L)$.  
For any $\gamma \in H_1(L)$, we have 
\begin{equation} 
\label{eq:Seidel-potential} 
\delta_{\gamma+\lambda,0} z^{\alpha_0}= 
\langle \hS^{(2)}, dW_\gamma \rangle 
+ \delta_{\gamma,0} \tS^{(0)} 
\end{equation} 
in $\Lambda^{\rm op}$, 
where $\tS^{(0)}$ and $\hS^{(2)}$ are in Definition \ref{def:Seidellift}, 
$W_\gamma$ is in Definition \ref{def:W-gamma} 
and $dW_\gamma$ is its logarithmic derivative: 
\[
dW_\gamma := \sum_{\alpha\in H_2(M,L):\, 
\mu(\alpha)=2, \ \partial \alpha = \gamma}  
\alpha \otimes n_\alpha z^\alpha  
\in H_2(M,L) \otimes \Lambda^{\rm op}. 
\]
Recall that $\alpha_0$ is the maximal disc class introduced 
before Proposition \ref{prop:retractionmap} and 
$\lambda\in H_1(L)$ is the class of an $S^1$-orbit on $L$. 
\end{theorem} 

Summing over all $\gamma\in H_1(L)$ in 
\eqref{eq:Seidel-potential}, we obtain: 
\begin{corollary} 
\label{cor:degeneration}
Assume that the degeneration formula (Conjecture \ref{conj:degeneration}) 
and Assumption \ref{assump:moduli} hold for $(M,L)$.  
Then we have 
\begin{equation} 
\label{eq:SdW} 
z^{\alpha_0} = 
\langle \hS^{(2)},  dW \rangle + \tS^{(0)} 
\quad 
\text{in  $\Lambda^{\rm op}$}.  
\end{equation} 
\end{corollary} 

Via the natural map $H^1(L) \to H^2(M,L)$, 
an element of $H^1(L)$ can be regarded as a vector field 
tangent to the fibre of the map 
$\Spec \Lambda^{\rm op} \to \Spec \Lambda$.  
We define the (relative) \emph{Jacobi algebra} 
of the potential $W$ as 
\[
\Jac(W) := \Lambda^{\rm op}/
\Lambda^{\rm op} \langle   H^1(L), dW \rangle  
\]
where $\Lambda^{\rm op}\langle   H^1(L), dW \rangle$ 
denotes the ideal of $\Lambda^{\rm op}$ generated 
by $\langle \varphi, dW \rangle$, 
$\varphi \in \Im(H^1(L) \to H^2(M,L))$. 
As a class in the Jacobi algebra, 
the right-hand side of \eqref{eq:SdW} 
depends only on the Seidel element $\tS$ itself, 
not on the lift $\hS^{(2)}$. 
We can interpret it 
as the derivative of the bulk-deformed potential 
$W + t^0$ with respect to $\tS$, where $t^0$ is a co-ordinate 
on $H^0(M)$. 
The derivative of $W+t^0$ defines 
the so-called \emph{Kodaira-Spencer mapping}: 
\[
\KS \colon H^{\le 2}(M)\otimes \Lambda \to \Jac(W).   
\]
Then the equation \eqref{eq:SdW} implies 
\[
\KS(\tS) = [z^{\alpha_0}] \qquad 
\text{in} \quad \Jac(W). 
\]

\begin{remark} 
Assumption \ref{assump:moduli} (i)--(iii) ensures that 
the conditions in Conjecture \ref{conj:degeneration} hold 
for all $\beta$ with $\mu(\beta) \le -2$. 
Using the formula \eqref{eq:degenerationformula-explicit} for 
$\beta$ with $\mu(\beta) \le -2$ and $\partial \beta = \lambda$, 
we find: 
\[
\sum_{d\in i_*H_2(L)} \cS_{\sigma+d} \cap [L] = 0 \quad 
\text{if} \quad 
\lan c_1^{\rm vert}(E), \sigma \ran \le -1. 
\]
This supports the validity of Assumption \ref{assump:moduli} (iv).  
\end{remark} 

\begin{remark} 
A more intuitive explanation for the formula \eqref{eq:SdW} 
is as follows. One can think of the moduli space 
$\cM_{1,1}(\beta)$ of stable holomorphic discs 
with boundaries in $L$ and with one interior and one boundary marked points 
as giving a correspondence between $M$ and the free loop space 
$\cL L = \Map(S^1, L)$ of $L$. This correspondence 
should give rise to a map (bulk-boundary map) 
\[
C_*(M) \to C_*(\cL L)
\]
of chain complexes. One can view this as an analogue of the Kodaira-Spencer map.  
One can speculate that this map is an intertwiner between
the Seidel homomorphism $S \colon C_*(M) \to C_*(M)$ 
and the map $\cL L \to \cL L $ induced by the $S^1$-action.  
\end{remark}

\section{Potential function of a semi-positive toric manifold}
Using the degeneration formula 
(Conjecture \ref{conj:degeneration}), we compute the 
potential function of a Lagrangian torus fibre of 
a semi-positive toric manifold $X$.  
This confirms a conjecture (now a theorem \cite{CLLT12}) 
of Chan-Lau-Leung-Tseng \cite{CLLT11}. 

\subsection{Toric manifolds} 
\label{subsec:toric}

We fix notation on toric geometry.  
For more details we refer the reader to \cite{Au-To04, CoKa-Mi99, CoLi-To10}.  
For this paper a \emph{toric manifold} $X$ 
is a smooth projective toric variety, as constructed 
from the following data.

\begin{enumerate}
\item An integral lattice $N \cong \Z^n$ and its
  dual $M=\Hom(N,\Z)$.  We denote by $\lan \cdot
  ,\cdot\ran$ the natural pairing between $N$ and
  $M$.
\item A fan $\Sigma$ in $N_\R:=N\otimes \R$
  consisting of a collection of strongly convex
  rational polyhedral cones $\sigma\subset N_\R$,
  which is closed under intersections and taking
  faces. 
\end{enumerate} 
In order for $X$ to be smooth and projective, 
we need to assume that $\Sigma$ is complete, regular 
and admits a strongly convex piecewise-linear function. 
Let $\Sigma(1)$ denote the set of
$1$-cones (rays) in $\Sigma$, and we let 
$b_1,\dots,b_m$
denote integral primitive generators of
the $1$-cones.  
The \emph{fan sequence} of $X$ is the exact sequence
\begin{equation}
\label{eq:fans}
\begin{CD} 
0 @>>> \bL @>>> \Z^m @>>> N @>>> 0,
\end{CD}
\end{equation}
where the third arrow takes the canonical basis
to the primitive generators $b_1,\dots, b_m\in N$ and $\bL$
is defined to be the kernel of the third arrow. 
The dual of the sequence \eqref{eq:fans} is the
\emph{divisor sequence}
\begin{equation}
\label{eq:divs}
\begin{CD}
  0 @>>>  M @>>> \Z^m @>{\kappa}>>  \bL^\vee @>>> 0.
\end{CD} 
\end{equation}
The second arrow takes $v\in M$ into the tuple
$(\lan b_i, v\ran)^m_{i=1}$. 
The third arrow is denoted by $\kappa\colon \Z^m \to \bL^\vee$. 

The fan sequence tensored with $\C^\times$ 
gives the exact sequence of tori: 
\begin{equation*}
\begin{CD} 
 1 @>>> \G @>>> (\C^\times)^m@>>> 
\T  @>>> 1 
\end{CD} 
\end{equation*}
with $\G := \bL \otimes \C^\times$ and $\T := N\otimes \C^\times$. 
Let the torus $\G$ act on $\C^m$ by the 
second arrow $\G \to (\C^\times)^m$. 
The combinatorics of the fan defines a stability condition 
of this action as follows. Let $Z(\Sigma)$ denote the
union
\begin{equation} 
\label{eq:del}
Z(\Sigma):=\bigcup_{I\notin \cA} \C^I, 
\quad \C^I=\{(x_1,\dots,x_m): x_i=0 \text{ for }
i\notin I\}, 
\end{equation} 
where $\cA$ is the collection of anti-cones, that
is the complements of the subsets of indices that yields a cone
in the fan
\[\cA:=\left\{I \ : \ \sum_{i\in \{1,\dots,m\} \setminus I}
  \R_{\geq 0}b_i \in \Sigma\right\}.\]
The toric variety $X$ is defined as the quotient
\begin{equation*}
X:=\cU_\Sigma /\G;\ \quad 
\cU_\Sigma:=\C^m \setminus Z(\Sigma).
\end{equation*}
The torus $\T = (\C^\times)^m/\G$ acts naturally 
on $X$. 
The toric manifold $X$ contains $\T$ as an open 
free orbit; $X$ is a compactification 
of $\T$ along the rays in $\Sigma(1)$. 

Each character $\xi:\G\to \C^\times$ defines a
line bundle
\[
L_{\xi}:=\C\times_{\xi,\G} \cU_\Sigma \to X. \]
The correspondence $\xi\mapsto L_\xi$ yields an
identification of the Picard group with the
character group of $\G$.  Thus, we have
\[
 \bL^\vee=\Hom(\G,\C^\times) \cong \Pic(X) 
\overset{c_1}{\cong} H^2(X;\Z). 
\]
The $i$-th toric divisor is given by 
\[
D_i := \{[x_1,\dots,x_m] \,:\, x_i=0\}\subset X.
\] 
The Poincar\'{e} dual of $D_i$ is the image 
$\kappa(e_i)\in \bL^\vee \cong H^2(X;\Z)$ 
of the standard basis $e_i \in \Z^m$ under 
the map $\kappa$ in \eqref{eq:divs}. 
By abuse of notation, $D_i$ sometimes also denotes 
the corresponding cohomology class $\kappa(e_i)$ in $H^2(X;\Z)$. 
We note that $\bL =H_2(X;\Z)$. 
The first Chern class $c_1(X)$ of $X$ is given by 
$D_1+ \cdots + D_m$. 

The \emph{K\"{a}hler cone} $C_X$ of $X$, 
the cone consisting of K\"ahler classes, is given by 
\[
C_X:=\bigcap_{I\in \cA} \sum_{i\in I} \R_{>0} \kappa(e_i) 
\subset \bL^\vee\otimes \R=H^2(X;\R).
\]
The cone $C_X$ is nonempty if and only if $X$ is projective.  
Set $r:= m-n$. 
We choose a nef integral basis $p_1,\dots, p_r$ of $H^2(X;\Z)$, 
that is an integral basis such that
$p_a\in \ol{C}_X$ for all $a=1,\dots,r$. 
Then we write the toric divisor classes as
\begin{equation}
\label{eq:div-basis}
D_j=\kappa(e_j) = \sum_{a=1}^r m_{aj}p_a, 
\end{equation}
for some integer matrix $(m_{aj})$.  
The \emph{Mori cone} $\NE(X)\subset H_2(X,\R)$ 
is the dual of the cone $\ol{C}_X$.  
We set $\NE(X)_\Z := \NE(X)\cap H_2(X;\Z)$.

The toric manifold $X$ can be alternatively 
defined as a symplectic quotient.   
Let $\G_\R \cong (S^1)^r$ be the maximal compact 
subgroup in $\G$. 
The $\G_\R$-action on $\C^m$ is generated by the
moment map
\[ 
\phi \colon \C^m \to \frg^\vee_\R, \ \ \
\phi(x_1,\dots, x_m)=\kappa(|x_1|^2,\dots,|x_m|^2) 
\]  
where $\kappa \colon \R^m\to \bL^\vee\otimes \R$ 
is the map in the divisor sequence \eqref{eq:divs} tensored with $\R$. 
For any K\"ahler class $\omega\in C_X$, we have a
diffeomorphism (\cite{Au-To04, Gu-Mo94})
\[
\phi^{-1}(\omega)/\G_\R \cong X. 
\]
The left-hand side is a symplectic quotient 
and is equipped with a reduced symplectic form. 
The cohomology \emph{class} of the reduced symplectic form 
coincides with $\omega$; by abuse of notation 
we let $\omega$ also denote the reduced symplectic \emph{form}.   

Let $\T_\R \cong (S^1)^n$ be the maximal compact 
subgroup of $\T$. 
The $\T_\R$-action on the symplectic toric manifold 
$(X,\omega)$ admits a moment map: 
\begin{align*} 
& \Phi_\omega  \colon X \longrightarrow \kappa^{-1}(\omega), \\ 
& \Phi_\omega([x_1,\dots,x_m]) = (|x_1|^2,\dots, |x_m|^2) 
\quad 
\text{with} 
\quad 
(x_1,\dots, x_m) \in \phi^{-1}(\omega), 
\end{align*} 
where the affine subspace $\kappa^{-1}(\omega)\subset \R^m$ can be 
identified with $M_\R = M\otimes \R \cong \frt_\R^\vee$ 
up to translation. 
The image of the moment map $\Phi_\omega$ is 
the convex polytope: 
\begin{align*}
P(\omega) & = \left\{ (t_1,\dots,t_m) \in \R^m 
\;:\; t_i \ge 0, \ \kappa(t_1,\dots,t_m) = \omega 
\right\} \\ 
& \cong
\left\{ v \in M_\R \; : \; \lan b_i, v\ran \le -c_i, \ i=1,\dots,m 
\right\}.   
\end{align*} 
In the second line, we took a lift $(c_1,\dots,c_m)$ 
of $\omega$ (such that $\omega = \kappa(c_1,\dots,c_m)$) 
to identify $\kappa^{-1}(\omega)$ with $M_\R$. 
This is called \emph{momentum polytope}. 
The facet $F_i\subset P$ of $P(\omega)$ normal to $b_i\in N$ 
corresponds to the toric divisor 
$D_i=\Phi_\omega^{-1}(F_i)\subset X$.

\subsection{Potential function of a Lagrangian torus fibre} 
\label{subsec:potential-toric} 
Cho-Oh \cite{ChOh-Fl06} calculated potentials of Lagrangian torus fibres 
for Fano toric manifolds and matched them up with 
mirror Landau-Ginzburg potentials of Givental and Hori-Vafa. 
Fukaya-Oh-Ohta-Ono \cite{FuOh-La10} studied 
potentials for general symplectic toric manifolds.   
Chan \cite{Chan-formula}, Chan-Lau \cite{ChLa-open} and 
Chan-Lau-Leung-Tseng \cite{CLLT11, CLLT12} have 
studied the potential functions for semi-positive toric 
manifolds by establishing an equality between open and 
closed Gromov-Witten invariants.

Let $X$ be a toric manifold in the previous section. 
Every free $\T_\R$-orbit in $X$ is a 
fibre of the moment map $\Phi_\omega \colon X \to P(\omega)$ 
of an interior point in $P(\omega)$, and vice versa.  
We call it a \emph{Lagrangian torus fibre} of $X$. 
For a Lagrangian torus fibre $L$, we have a 
homotopy exact sequence: 
\begin{equation} 
\label{eq:discs}
\begin{CD}
 0 @>>> \pi_2(X)  @>>> 
\pi_2(X,L) @>{\partial} >>   \pi_1(L) @>>> 0.
\end{CD} 
\end{equation} 
Let $\beta_i \in \pi_2(X,L)$ denote the class 
represented by the holomorphic disc 
$u_i \colon \D \to X$: 
\begin{equation} 
\label{eq:representative-betai} 
u_i ( z ) = [c_1,\dots,c_{i-1},c_i z, c_{i+1},\dots, c_m], \quad 
|z| \le 1  
\end{equation} 
where $[c_1,\dots,c_m]\in X$ is a point on the Lagrangian $L$ 
(thus $c_i \neq 0$ for all $i$). 
The class $\beta_i$ intersects with toric divisors as 
\[
\beta_i \cdot D_j = \delta_{ij}.  
\]
The relative homotopy group $\pi_2(X,L)$ 
is an abelian group freely generated 
by the classes $\beta_1,\dots,\beta_m$ 
and the toric divisors $D_1,\dots,D_m$ define a 
dual basis of $H^2(X,L)$. 
Under the identification: 
\[
\pi_2(X) \cong H_2(X;\Z) \cong \bL, \quad 
\pi_2(X,L) \cong H_2(X,L;\Z) \cong \Z^m, \quad 
\pi_1(L) \cong N 
\]
the exact sequence \eqref{eq:discs} above can 
be identified with the fan sequence \eqref{eq:fans}, i.e.\ 
$\partial \beta_i = b_i$. 
The \emph{Maslov index} 
\[
\mu\colon \pi_2(X,L) \longrightarrow \Z 
\]
is given by the intersection with $2 (D_1+\cdots + D_m) \in H^2(X,L)$ 
\cite[Theorem 5.1]{ChOh-Fl06}. 

We consider the potential function (Definition \ref{def:W-general}) 
of a Lagrangian torus fibre $L\subset X$. 
As before, let $\cM_1(\beta)$ denote the moduli space of 
bordered stable maps to $(X,L)$ in the class $\beta \in \pi_2(X,L)$ 
with one boundary marked point. 

\begin{proposition} 
\label{prop:effdisc-semipos}
Suppose that $c_1(X)$ is semi-positive. 
Then $\cM_1(\beta)$ is empty for all $\beta$ 
with $\mu(\beta)\le 0$. 
If $\cM_1(\beta)$ is non-empty for $\beta$ with 
$\mu(\beta)=2$, then $\beta = \beta_i + d$ 
for some $i$ and $d\in \NE(X)_\Z$ such that 
$\lan c_1(X), d\ran =0$. 
\end{proposition} 
\begin{proof} 
Let $\beta$ be a class of a bordered stable map to $(X,L)$. 
By the classification of holomorphic discs by Cho-Oh \cite{ChOh-Fl06}, 
we find that $\beta$ is of the form: 
\begin{equation} 
\label{eq:effdisc}
\beta = \sum_{i=1}^m k_i \beta_i + d 
\end{equation} 
for some $k_i \ge 0$ and $d \in \NE(X)_\Z$. 
Here $\sum_{i=1}^m k_i \beta_i$ is the degree of 
disc components and $d$ is the degree  of sphere bubbles. 
Hence $\mu(\beta) = 2\sum_{i=1}^m k_i + 2 \lan c_1(X),d\ran \ge 0$. 
We claim that $(k_1,\dots,k_m)=0$ implies $\mu(\beta) \ge 4$. 
If $(k_1,\dots,k_m)=0$, a bordered stable map 
of class $\beta$ is the union of a \emph{constant} 
disc and sphere bubbles. 
In this case, at least one non-trivial sphere component has to touch $L$. 
Let $d_1$ be the degree of a non-trivial sphere component 
touching $L$ and let $d_2$ be the degree of the  
remaining sphere bubbles. Then $d = d_1 + d_2$ with 
$d_1,d_2 \in \NE(X)_\Z$. Since $D_i$ is disjoint from $L$, 
we have $\lan D_i, d_1\ran \ge 0$. 
Since $d_1 \neq 0$, we have $\sum_{i=1}^m \lan D_i, d_1 \ran \ge 1$. 
Also it is impossible that $\sum_{i=1}^m \lan D_i, d_1 \ran =1$ 
since $d_1$ gives the relation 
$\sum_{i=1}^m \lan D_i, d_1 \ran b_i =0$ in $N$. 
Thus 
\[
\mu(\beta) = 2 \lan c_1(X), d\ran 
\ge 2 \sum_{i=1}^m \lan D_i, d_1 \ran \ge 4. 
\]
The claim follows. Consequently, $\mu(\beta)\le 2$ implies 
$(k_1,\dots,k_m) \neq 0$. The proposition follows easily. 
\end{proof} 

In particular, the potential function of a Lagrangian 
torus fibre (Definition \ref{def:W-general}) is well-defined 
for a semi-positive toric manifold. 
\begin{remark} 
\label{rem:W-nonsemipos}
Fukaya-Oh-Ohta-Ono \cite{FuOh-La10} defined the 
potential function of a Lagrangian torus fibre 
even without semi-positivity assumption. 
They defined virtual cycles and open Gromov-Witten invariants 
$n_\beta\in \Q$ for all $\beta$ with $\mu(\beta)=2$ 
using $\T_\R$-equivariant perturbations, see  
\cite[Lemmata 11.2, 11.5, 11.6, 11.7]{FuOh-La10}. 
In general, since every effective stable disc class $\beta$ 
is of the form \eqref{eq:effdisc}, the potential $W$ lies 
in the completed group ring: 
\[
\Q[\![ (\Z_{\ge 0})^m + \NE(X)_\Z]\!] \subset \Lambda^{\rm op} 
\]
where $\NE(X)_\Z$ is regarded as a subset of $\Z^m$ 
via the second arrow in the fan sequence \eqref{eq:fans}. 
Notice that $(\R_{\ge 0})^m + \NE(X)$ is a strictly 
convex cone. 
\end{remark} 

\begin{example}[\cite{ChOh-Fl06}]
\label{exa:basicdisc} 
When $\beta=\beta_i$, the moduli space $\cM_1(\beta_i)$ consists 
of holomorphic discs of the form \eqref{eq:representative-betai} 
and $\ev\colon \cM_1(\beta_i) \cong L$; moreover all such discs are 
Fredholm regular \cite[Theorem 6.1]{ChOh-Fl06}. 
Therefore we have $n_{\beta_i} =1$. 
\end{example} 

We write 
\[
z^\beta = z_1^{k_1}z_2^{k_2} \cdots z_m^{k_m} \in \Q[H_2(X,L;\Z)]
\]
for $\beta = k_1 \beta_1 + \cdots + k_m \beta_m$. 
Also we write 
\[
q^d = q_1^{\lan p_1, d\ran} q_2^{\lan p_2,d\ran} 
\cdots q_r^{\lan p_r, d\ran} \in \Q[H_2(X;\Z)] 
\]
for $d \in H_2(X;\Z)$, where $p_1,\dots,p_r$ is the nef integral 
basis of $H^2(X;\Z) \cong \bL^\vee$ we chose in \S \ref{subsec:toric}. 
Note that we have a natural inclusion of the group rings: 
\[
\Q[H_2(X;\Z)] \hookrightarrow \Q[H_2(X,L;\Z)]. 
\]
By this we identify $q^d$ with $z^d$; in co-ordinates: 
\begin{equation} 
\label{eq:q-z} 
q^d = z^d = z_1^{\lan D_1, d\ran} z_2^{\lan D_1, d\ran} 
\cdots z_m^{\lan D_m, d\ran} \quad 
\text{or} \quad q_a = \prod_{i=1}^m z_i^{m_{ai}}
\end{equation} 
where  $(m_{ai})$ is the divisor matrix in \eqref{eq:div-basis}. 
Using these notations and Proposition \ref{prop:effdisc-semipos}, 
we can write the potential function of $(X,L)$ in the 
following form when $c_1(X)$ is semi-positive. 

\begin{definition} 
\label{def:correctionterm} 
Let $X$ be a semi-positive toric manifold. 
We present the potential function $W$ of a Lagrangian 
torus fibre in the form: 
\[
W = w_1+ \cdots + w_m 
\]
where $w_i = f_i(q) z_i$ and 
\[
f_i(q) = \sum_{d\in \NE(X)_\Z \, :\, \lan c_1(X), d\ran =0 } 
n_{\beta_i + d} q^d.  
\] 
We call $f_i(q)$ the \emph{correction term}. 
This decomposition of $W$ 
is parallel to Definition \ref{def:W-gamma}.
\end{definition}

Note that we have $f_i(q) = 1+ O(q)$ by Example \ref{exa:basicdisc}. 
The correction term $f_i(q)$ was denoted by $1+ \delta_i(q)$ 
in \cite{CLLT11}. When $X$ is Fano, all the correction 
terms are $1$ and  
\[
W = z_1+ \cdots + z_m. 
\]
This is the result of Cho-Oh \cite{ChOh-Fl06}. 
By the \emph{fan polytope}, we mean the convex hull of the 
ray vectors $b_1,\dots, b_m \in N_\R$.  
In the proof of \cite[Corollary 4.12]{CLLT11}, Chan-Lau-Leung-Tseng 
showed the following: 
\begin{proposition}[Chan-Lau-Leung-Tseng \cite{CLLT11}] 
\label{prop:vertex-vanishing} 
Let $f_i(q)$ be the correction terms of the potential 
of a semi-positive toric manifold $X$. 
If the vector $b_i$ is a vertex of the fan
polytope of $X$, then $f_i(q)=1$.
\end{proposition}

\subsubsection{Open-closed moduli space} 
\label{subsubsec:open-closed}
We explain that the potential $W$ of a Lagrangian torus fibre 
can be interpreted as a formal function 
on the \emph{open-closed moduli space} introduced below. 

The \emph{closed moduli space} $\MM_{\text{cl}}$ 
of $X$ is defined to be: 
\[
\MM_{\text{cl}}=\{ \exp(-\omega+iB)\in \bL^\vee \otimes
\C^\times \, :\, \omega,B\in \bL^\vee\otimes \R, 
\ \omega\in C_X \}. 
\]
This is also called the 
\emph{complexified K\"{a}hler moduli space}.  
The nef basis $p_1,\dots,p_r$ of $\bL^\vee\cong H^2(X;\Z)$ 
in \S \ref{subsec:toric} defines 
$\C^\times$-valued co-ordinates $(q_1,\dots, q_r)$ 
on $\MM_{\text{cl}} \subset \bL^\vee \otimes \C^\times$. 

The \emph{open-closed moduli space} 
$\MM_{\text{opcl}}$ is defined to be 
the set of triples $(q, L,  h)$ such that 
\begin{itemize}
\item a closed moduli $q = \exp(-\omega + i B)\in 
\MM_{\text{cl}}$; 
\item a Lagrangian torus fibre $L =L_\eta= \Phi_\omega^{-1}(\eta)$ 
at $\eta \in P(\omega)^\circ$; 
\item a class $h \in H^2(X, L ; U(1))$ 
which maps to $\exp(i B) \in H^2(X;U(1))$.  
\end{itemize} 
When the $B$-field vanishes $B=0$,  the class 
$h$ defines a $U(1)$-local system on $L$ 
via the exact sequence: 
\[
\begin{CD} 
0@>>> H^1(L; U(1)) @>>> H^2(X, L; U(1)) 
@>>> H^2(X; U(1)) @>>> 0.
\end{CD}
\]
Let $\eta =(\eta_1,\dots, \eta_m)\in \R^m$ be the co-ordinates of 
$\eta$ and write $h =(h_1,\dots, h_m)$ using 
the identification $H^2(X,L; U(1)) \cong (S^1)^m$; 
and set 
\begin{equation} 
\label{eq:coordinate-z} 
z_i := \exp(-\eta_i)  h_i.  
\end{equation} 
The parameter $z=(z_1,\dots,z_m)$ here determines 
$\eta_i\in \R$, $h_i \in S^1$ by polar decomposition; 
then $\eta$ determines $\omega$ by the condition 
$\eta \in P(\omega)^\circ$ (as $\omega = \kappa(\eta)$) 
and $h$ determines $\exp(i B)$. 
Thus $z$ determines a point of $\MM_{\text{opcl}}$. 
We have: 
\[
\MM_{\text{opcl}} \cong \left\{z =(z_1,\dots, z_m)
\in (\C^\times)^m \,:\, 
|z_i|<1 \text{ for all $i$}, \ \kappa_{\C^\times}(z) 
\in \MM_{\text{cl}} \right\}  
\]
where $\kappa_{\C^\times} \colon 
(\C^\times)^m \to \bL^\vee \otimes \C^\times$ 
is the third arrow of the divisor sequence \eqref{eq:divs} 
tensored with $\C^\times$. 
A point $z = (z_1,\dots,z_m)$ of the right-hand side 
parametrizes  
\begin{itemize} 
\item 
a closed moduli $q = \exp(-\omega + iB) = \kappa_{\C^\times} (z)$;  

\item 
a Lagrangian torus fibre $L = L_\eta$ at 
$\eta = (-\log |z_1|,\dots, - \log |z_m|) 
\in P(\omega)^\circ$; 

\item 
a class $h=(z_1/|z_1|,\dots,z_m/|z_m|)\in H^2(X,L_\eta)$ which is a lift of 
$\exp(i B)$. 
\end{itemize}
We regard $W$ as a formal function on $\MM_{\text{opcl}}$ via 
these co-ordinates $(z_1,\dots, z_m)$. 
The open-closed moduli is fibred over $\MM_{\text{cl}}$: 
\[
\pi \colon \MM_{\text{opcl}} \to \MM_{\text{cl}}, \quad 
z \mapsto \kappa_{\C^\times}(z).  
\]
By pulling-back the co-ordinates $q_1,\dots,q_r$ by $\pi$, 
we obtain the same relation between $z_i$ and $q_a$ 
as in \eqref{eq:q-z}. 
The fibre $\MM_{\text{opcl},q}= \pi^{-1}(q)$ 
has the structure of an $(M_\R/M) \cong (S^1)^n$-bundle 
over $P(\omega)^\circ$ via the map: 
\[
\MM_{\text{opcl},q} \to P(\omega)^\circ, \quad 
(z_1,\dots, z_m) \mapsto \eta = 
(- \log |z_1|, \dots, - \log |z_m|).  
\] 
This is a torus fibration dual to the moment map 
$\Phi_\omega \colon X\to P(\omega)$; we can 
view it as a mirror of $(X,q)$. 

\begin{proposition} 
Via the co-ordinates $(z_1,\dots,z_m)$ on $\MM_{\rm opcl}$, 
the potential function of a Lagrangian torus fibre is identified 
with the following formal sum of functions on $\MM_{\rm opcl}$: 
\begin{equation*}
 W(q,L,h) = \sum_{
\beta \in \pi_2(X,L) \, : \, \mu(\beta)=2 } 
  n_\beta h(\beta)  e^{-\int_{\beta}\omega} 
\end{equation*}
where $q = \exp(-\omega+ i B)$. 
\end{proposition}
\begin{proof} 
For $\beta = \beta_i$, we have 
(see \cite[Theorem 8.1]{ChOh-Fl06}) 
\[
h(\beta_i) = \frac{z_i}{|z_i|}, \quad 
\int_{\beta_i} \omega = \eta_i = - \log |z_i| 
\]
and thus $h(\beta_i) e^{-\int_{\beta_i} \omega} = z_i$ 
(cf.\ \eqref{eq:coordinate-z}). 
Therefore $h(\beta) e^{-\int_{\beta} \omega}= z^\beta$ 
for every $\beta$. 
\end{proof} 

\begin{remark} 
When $B=0$, the term $h(\beta)$ is the holonomy along 
the loop $\partial \beta\in \pi_1(L)$ of the $U(1)$-local 
system associated to $h$. 
This matches with the usual interpretation.  
In general, this term cannot be 
interpreted just as holonomy. 
\end{remark} 

Fukaya-Oh-Ohta-Ono \cite[Theorem 2.32]{FuOh-La10A} 
showed that the Jacobi algebra of the potential 
function restricted to the fibre $\MM_{\text{opcl},q} = \pi^{-1}(q)$ 
is isomorphic to the quantum cohomology ring 
of $(X,q)$ in a certain $q$-adic sense.

\subsection{Seidel elements for toric varieties 
and Givental's mirror transformation}  
\label{subsec:Seidel-toric} 
We review our previous computation \cite{GoIr-Se12} 
relating Seidel elements for toric varieties 
to Givental's mirror transformation \cite{Gi-A-98}. 
Let $X$ be a toric manifold from \S \ref{subsec:toric} 
with $c_1(X)$ semi-positive. 

\subsubsection{Seidel elements associated to 
the $\C^\times$-actions fixing toric divisors}

For each toric divisor $D_j$ of $X$, 
we can associate a $\C^\times$-action $\rho_j$ 
on $X$ rotating around $D_j$. 
It is given by: 
\[
\rho_j(\lambda) \colon 
[x_1,\dots, x_m]\longmapsto 
[x_1,\dots,\lambda^{-1}x_j, \dots, x_m], \qquad 
\lambda \in
\C^\times.  
\]
The toric divisor $D_j=\{x_j=0\}$ 
is the maximal fixed component of this action. 
Let $E_j$ denote the associated bundle of 
this $\C^\times$-action and let 
$S_j$ denote the corresponding Seidel element. 
We also write $S_j = q_0 \tS_j$ with $\tS_j \in QH^*(X)$ 
following Definition \ref{def:seidel}. 
Using the Seidel representation (see Remark \ref{rem:Seidelrep}), 
McDuff-Tolman \cite{McTo-To06} showed the 
following multiplicative relations 
in $QH(X)[q^{-d}\,:\, d\in \NE(X)_\Z]$:  
\begin{equation} 
\label{eq:Seidel-mult}
\prod_{j=1}^m \tS_j^{\lan D_j, d\ran} =q^d \qquad 
\text{for } \ d\in H_2(X;\Z). 
\end{equation} 

\subsubsection{Givental's mirror theorem} 
Givental \cite{Gi-A-98} introduced 
the two cohomology-valued functions  
\begin{align*} 
I(y,z) & =e^{\sum_{i=1}^r p_i\log y_i/z} 
\sum_{d\in \NE(X)_\Z} \prod_{i=1}^m 
\left(
\frac{\prod_{k=-\infty}^0
  (D_i+kz)}{ \prod_{k=-\infty}^{\lan D_i,
    d\ran}(D_i+kz)}
\right)
y^d \\ 
J(q,z) & = e^{\sum_{i=1}^r p_i \log q_i/z} 
\left( 1 + 
\sum_{j}  
\sum_{d\in \NE(X)_\Z \setminus\{0\}} 
\lan \frac{\phi_j}{z(z-\psi)} \ran^X_{0,1,d} 
\phi^j q^d\right) 
\end{align*} 
called the \emph{$I$-function} and the \emph{$J$-function} 
respectively.  
Here we used a nef basis $\{p_1,\dots,p_r\} \subset H^2(X)$ 
in \S \ref{subsec:toric} and write
\[
q^d = q_1^{\lan p_1, d\ran} \cdots q_r^{\lan p_r, d\ran},  
\qquad 
y^d = y_1^{\lan p_1, d\ran} \cdots y_r^{\lan p_r, d\ran}, 
\]
and $\{\phi_j\}$ and $\{\phi^j\}$ are mutually 
dual basis of $H^*(X)$. 
The variables $y =(y_1,\dots,y_r)$ are called 
\emph{mirror co-ordinates}, 
i.e.\ co-ordinates of the complex moduli of the mirror Landau-Ginzburg 
model. 
Givental \cite{Gi-A-98} showed the following \emph{mirror theorem}: 

\begin{theorem}[\cite{Gi-A-98}]  
We have $I(y,z) = J(q,z)$ under a change of coordinates 
of the form 
$\log q_i= \log y_i + g_i(y)$, $i=1,\dots,r$, 
$g_i(y) \in \Q[\![y_1,\dots,y_r]\!]$ with $g_i(0)=0$. 
The functions $g_i(y)$ here 
are uniquely determined by the asymptotics: 
\[
I(y,z) = e^{\sum_{i=1}^r p_i \log y_i/z} \left(1 +
  \sum_{i=1}^r g_i(y) \frac{p_i}{z} + o(z^{-1})
\right).
\]
\end{theorem} 

The change of co-ordinates is 
called \emph{mirror transformation} (or \emph{mirror map}). 

\subsubsection{Batyrev elements and Seidel elements} 
\label{subsubsec:Bat-Sei}
In \cite{GoIr-Se12}, we introduced \emph{Batyrev elements} 
$\tD_j$, $j=1,\dots,m$. They are defined by 
\[
\tD_j := \sum_{a=1}^r m_{aj} \tp_a, \quad 
\tp_a := \sum_{b=1}^r \parfrac{\log q_b}{\log y_a} p_b.  
\]
Note that $\tD_j$ is an element corresponding to 
the vector field $\sum_{a=1}^r m_{aj} y_a \partial/\partial y_a$ 
whereas the genuine divisor class $D_j$ corresponds to 
the vector field $\sum_{a=1}^r m_{aj} q_a \partial/\partial q_a$ 
(see \eqref{eq:div-basis}). 
Batyrev elements are determined by, and 
determine the Jacobi matrix $(\partial \log q_b/\partial \log y_a)$ 
of the mirror transformation. 
Using Givental's mirror theorem, we find that 
the Batyrev elements satisfy the multiplicative relations  
(see \cite[Proposition 3.8]{GoIr-Se12}) 
\[
\prod_{j=1}^m \tD_j^{\lan D_j, d\ran} = y^d 
\qquad 
d\in H_2(X;\Z) 
\]
in the quantum cohomology ring.  
These are very similar to the multiplicative relations \eqref{eq:Seidel-mult} 
of Seidel elements, but note that co-ordinates $q$ are 
replaced with mirror co-ordinates $y$. 
Moreover, the Batyrev elements satisfy the following 
\emph{linear relations}: 
\begin{equation} 
\label{eq:linear_relation} 
\sum_{i=1}^m c_j \tD_j =0 \quad 
\text{whenever} \quad 
\sum_{i=1}^m c_j D_j = 0.  
\end{equation} 
The linear relations are obvious from the definition. 
These multiplicative and linear relations show that $\tD_j$ 
satisfy the relations of Batyrev's quantum ring \cite{Ba-Qu93}. 
It turns out that the Seidel elements are multiples of the Batyrev elements. 

\begin{theorem}[{\cite[Theorem 1.1]{GoIr-Se12}}]  
\label{thm:g0}
Let $g_0^{(j)}(y)$ be the following hypergeometric series 
in mirror co-ordinates: 
\begin{equation} 
\label{eq:g0j} 
g_0^{(j)}(y_1,\dots,y_r) 
=\sum_{\substack{\lan c_1(X), d\ran =0, 
        \lan D_j, d\ran<0\\
        \lan D_i, d\ran \geq
        0\ \text{for all } i\neq j}}
    \frac{(-1)^{\lan D_j, d\ran } 
     \left(- \lan D_j, d\ran -1\right)!}
{\prod_{i\neq j}\lan D_i, d\ran!}y^d.
\end{equation} 
Then under the mirror transformation we have 
\[
\tS_j = \exp\left(-g_0^{(j)}(y)\right) \tD_j. 
\]
\end{theorem} 

Conversely, one can recover the Batyrev elements from 
the Seidel elements in the following way. 

\begin{theorem}[{\cite[Theorem 1.2]{GoIr-Se12}}]
\label{thm:GI2} 
Given the Seidel elements $\tS_1,\dots,\tS_m$, 
the Batyrev elements 
$\tD_j \in H^*(X)\otimes \Q[\![q_1,\dots,q_r]\!]$, 
$j=1,\dots,m$ are uniquely characterized by  
the following conditions: 
\begin{enumerate}
\item $\tD_j = H_j \tS_j$ for some $H_j\in \Q[\![q_1,\dots,q_r]\!]$; 
\item $\tD_j = \tS_j$ if $b_j$ is a vertex of the fan polytope; 
\item $\tD_j$ satisfy the linear relations 
\eqref{eq:linear_relation}. 
\end{enumerate} 
In particular, the Seidel elements determine 
the mirror transformation $q\mapsto y$ 
and the functions $g_0^{(j)}(y)$. 
\end{theorem} 

\subsection{Correction terms of potential functions and 
Seidel elements}
Chan-Lau-Leung-Tseng \cite{CLLT11} gave a 
conjecture relating the correction terms of 
the potential function and the Seidel elements 
for a semi-positive toric manifold. 
\begin{conjecture}[{\cite[Conjecture 5.2]{CLLT11}}]
\label{conj:CLLT} 
For a semi-positive toric manifold, 
the correction term $f_j(q)$ of the potential function 
(Definition \ref{def:correctionterm}) coincides with 
$\exp(g_0^{(j)}(y))$ in Theorem \ref{thm:g0} 
under mirror transformation.  
\end{conjecture}

Originally Chan-Lau-Leung-Tseng \cite{CLLT11} 
proved this conjecture under the convergence assumption for $W$ 
using an isomorphism \cite{FuOh-La10A} of Jacobi ring and quantum cohomology. 
Recently they gave an alternative proof \cite{CLLT12} 
which does not require the convergence assumption. 
They identified open Gromov-Witten invariants  
with certain closed Gromov-Witten invariants of the 
associated bundle $E'_i$ given by the inverse 
$\C^\times$-action $\rho_i^{-1}$. 
They used the fact that a bordered stable map 
to $(M,L)$ with boundary class $b_i \in N \cong H_1(L)$ 
can be completed to a holomorphic sphere in the 
associated bundle $E_i'$. 
This is closely related to the fact that the central fibre 
$\ocE_0$ of the closing in \S \ref{ss:def}
is the union of the two associated bundles $E$ and $E'$ 
which correspond to mutually inverse $\C^\times$-actions. 

\subsection{Degeneration formula for toric manifolds} 
\begin{proposition} 
\label{prop:toric-assump} 
Assumption \ref{assump:moduli} holds for a pair $(X,L)$ 
equipped with the $\C^\times$-action $\rho_j$ around the 
prime toric divisor $D_j$  
we considered in \S \ref{subsec:Seidel-toric}. 
\end{proposition} 
\begin{proof} 
The statement (i) is shown in Proposition \ref{prop:effdisc-semipos} 
and (ii), (iii) are obvious. 
To verify the statement (iv), it is enough to show that 
every stable map $u\colon C\to E_j$ representing 
a class $\sigma\in H_2^{\rm sec}(E_j)$ 
with $\lan c_1^{\rm vert}(E_j), \sigma \ran = -1$ 
is contained in $\bigcup_{i=1}^m \hD_i$, where 
\[
\hD_i = D_i\times (\C^2\setminus \{0\})  \big/\C^\times
\]
is a toric divisor of $E_j$. 
Let $C = \bigcup C_\alpha$ be an irreducible decomposition 
of $C$. If $u_*[C_\alpha]$ is a section class, we have  
$\lan c_1^{\rm vert}(E_j), u_*[C_\alpha] \ran \ge -1$ 
by \eqref{eq:split-sigma0} and the semi-positivity of $c_1(X)$. 
If $u_*[C_\alpha]$ is not a section class, $u(C_\alpha)$ is 
contained in a fibre $X$ and we have 
$\lan c_1^{\rm vert}(E_j), u_*[C_\alpha] \ran = 
\lan c_1(X), u_*[C_\alpha] \ran \ge 0$ again by
the semi-positivity. 
Since $\lan c_1^{\rm vert}(E_j), \sigma \ran =-1$, 
we have 
\[
\lan c_1^{\rm vert}(E_j), u_*[C_\alpha] \ran = 
\begin{cases} 
-1 & \text{if $u(C_\alpha)$ is a section;} \\ 
0  & \text{otherwise.} 
\end{cases} 
\]
Suppose that $u(C)\not \subset \bigcup_{i=1}^m \hD_i$. 
Then we can find a component $C_\alpha$ such that 
$u(C_\alpha)$ is not a point and $u(C_\alpha) \not \subset 
\bigcup_{i=1}^m \hD_i$. 
Then $\langle \hD_i , u_*[C_\alpha]\rangle \ge 0$ for all $i$. 
Note that $\sum_{i=1}^m \hD_i$ is the Poincar\'{e} dual 
of $c_1^{\rm vert}(E_j)$. 
By the above calculation we see that  
$\langle \hD_i , u_*[C_\alpha] \rangle =0$ for all $i$ 
and $u(C_\alpha)$ is contained in a certain fibre $X$. 
Then $\langle D_i, u_*[C_\alpha] \rangle =0$ for 
all $i$. A homology class $d\in H_2(X)$ satisfying 
$\langle D_i, d \rangle =0$ for all $i$ is zero.  
This is a contradiction. 
\end{proof} 

Recall from Remark \ref{rem:W-nonsemipos}
that the potential function $W = W(z_1,\dots, z_m)$ of 
a toric manifold $X$ is an element of 
\[
R:= \Q[\![\NE(X)_\Z + (\Z_{\ge 0})^m ]\!] \subset \Lambda^{\rm op}. 
\]
We also set 
\[
K := \Q[\![\NE(X)_\Z]\!] \subset \Lambda. 
\]
Then $R$ is a $K$-algebra (cf.\ \eqref{eq:q-z}). 
For $f\in R$, we write (following notation in 
Theorem \ref{thm:degeneration}): 
\[
df = \left( 
z_1\parfrac{f}{z_1},\dots, z_m\parfrac{f}{z_m} 
\right ) \in \Z^m\otimes R \cong H_2(X,L) \otimes R. 
\]
In other words, 
\[
d z^\beta = \beta \otimes z^\beta 
\]
for $\beta \in H_2(X,L)$. 

We apply Theorem \ref{thm:degeneration} to 
the $\C^\times$-action $\rho_j$ rotating around $D_j$. 
Note that the $k$-th term $w_k$ of the potential $W$ in 
Definition \ref{def:correctionterm} 
corresponds to the boundary class $b_k \in N \cong H_1(L)$ 
and $w_k = W_{b_k}$ in the notation of 
Definition \ref{def:W-gamma}.  
Since the Seidel element $\tS_j$ in \S \ref{subsec:Seidel-toric} 
belongs to $H^2(X)\otimes K$, we have $\tS_j^{(0)}=0$ 
and $\tS_j = \tS_j^{(2)}$.  
By Proposition \ref{prop:toric-assump}, we can define the lift 
\[
\hS_j\in H^2(X,L)\otimes K
\] 
of $\tS_j = \tS_j^{(2)}$ 
as in Definition \ref{def:Seidellift}. 
The class $\lambda$ of an $S^1$-orbit on $L$ is 
$-b_j\in H_1(L)$ 
and the maximal disc class $\alpha_0$ 
is $\beta_j$. Hence we obtain: 

\begin{theorem}
\label{thm:degeneration-toric} 
Assume that the degeneration formula 
(Conjecture \ref{conj:degeneration}) holds 
for $(X,L)$ equipped with the $\C^\times$-action $\rho_j$ 
around the toric divisor $D_j$ (see \S \ref{subsec:Seidel-toric}). 
Then we have 
\begin{equation} 
\label{eq:degenerationformula-toric} 
\langle  \hS_j, dw_k\rangle =\delta_{jk} z_j. 
\end{equation} 
In particular, we have $\langle \hS_j, dW \rangle = z_j$. 
\end{theorem}

\subsubsection{Example} 
Consider the second Hirzebruch surface $\F_2 = \P(\cO_{\P^1}(-2) 
\oplus \cO_{\P^1})$, a compactification of $\cO_{\P^1}(-2)$. 
The divisor matrix \eqref{eq:div-basis} is: 
\[
(m_{ai}) = 
\begin{bmatrix} 
0 & -2 & 1 & 1 \\ 
1 & 1 & 0 & 0 
\end{bmatrix}.  
\]
The column vectors give toric divisors classes $D_1,D_2,D_3,D_4$. 
Here $D_1$ is the $\infty$-section, $D_2$ is the zero-section ($-2$ curve) 
and $D_3$, $D_4$ are fibres.  
The potential function has been calculated by 
Auroux \cite{Au-Sp09}, Fukaya-Oh-Ohta-Ono \cite{FOOO-toricdeg} 
and Chan-Lau \cite{ChLa-open}: 
\[
W = z_1 + (1+q_1) z_2 + z_3 + z_4. 
\]
Therefore we have 
\[
\begin{bmatrix}
d w_1 \\ d w_2 \\  d w_3 \\  d w_4 
\end{bmatrix} 
= 
\begin{bmatrix} 
z_1 &  0 & 0 & 0 \\ 
0  &  (1 -q_1)z_2 & q_1 z_2 & q_1 z_2 \\ 
0  & 0  & z_3 & 0 \\
0 & 0 & 0 & z_4 
\end{bmatrix} 
\]
where we used $q_1 = z_2^{-2} z_3 z_4$ (see \eqref{eq:q-z}) 
and $d (q_1 z_2)  = [
0,  
- q_1z_2,   
q_1 z_2,  
q_1z_2 ]   
$.  
Assuming the degeneration formula \eqref{eq:degenerationformula-toric}, 
we obtain 
\[
\left[\hS_1,\hS_2,\hS_3,\hS_4\right] = 
\left[ D_1, D_2, D_3, D_4 \right] 
\begin{bmatrix} 
1 & 0 & 0 & 0 \\
0 & \frac{1}{1-q_1} & -\frac{q_1}{1-q_1} & -\frac{q_1}{1-q_1} \\ 
0 & 0 & 1 & 0 \\ 
0 & 0 & 0 & 1 
\end{bmatrix}.  
\]
This is compatible with the calculations of $\tS_j$ 
by McDuff-Tolman \cite{McTo-To06} and 
Gonz\'{a}lez-Iritani \cite{GoIr-Se12}.  


\subsection{Kodaira-Spencer map} 
\label{ss:ksmap}
Recall from Definition \ref{def:correctionterm} 
that $w_i = f_i(q) z_i$ for some $f_i(q) \in K$. 
We have (using \eqref{eq:q-z}) 
\[
z_i \parfrac{w_j}{z_i}  = 
\left(\delta_{ij} + z_i\parfrac{f_j(q)}{z_i} \right )z_j 
= \left (\delta_{ij} + 
\sum_{a=1}^r m_{ai} q_a \parfrac{f_j}{q_a}(q)\right ) 
z_j \in  Kz_j. 
\]
Therefore we have an isomorphism of $K$-modules: 
\[
\frks \colon H^2(X,L) \otimes K 
\xrightarrow{\ \cong\ } \bigoplus_{j=1}^m 
K  z_j, 
\quad 
D_i \longmapsto \left( z_i\parfrac{w_1}{z_i}, \dots, 
z_i\parfrac{w_m}{z_i} \right).  
\]
The degeneration formula \eqref{eq:degenerationformula-toric} 
says that $\frks(\hS_i) = z_i$. 
For $\varphi \in H^1(L) = M$, we have 
\begin{align*} 
\frks(\delta\varphi) 
& = 
\bigoplus_{j=1}^m 
\sum_{i=1}^m \langle \varphi, b_i \rangle 
z_i \parfrac{w_j}{z_i}  
= 
\bigoplus_{j=1}^m 
\sum_{i=1}^m \langle \varphi, b_i \rangle 
\left( z_i \delta_{ij} f_j(q)+ z_i z_j \parfrac{f_j(q)}{z_i} \right) 
\\
& =  \bigoplus_{j=1}^m 
\lan \varphi, b_j \ran w_j
\in \bigoplus_{j=1}^m K z_j, 
\end{align*} 
where $\delta\colon H^1(L) \cong M \to H^2(X,L)\cong \Z^m$ 
is a coboundary map. 
Hence $\frks$ induces an isomorphism  
\[
\ks \colon H^2(X) \otimes K
\overset{\cong}{\longrightarrow}  
\bigoplus_{j=1}^m K z_j 
\Big/ 
\left \langle 
\textstyle 
\bigoplus_{j=1}^m \lan \varphi, b_j\ran w_j 
\,:\, \varphi \in M 
\right \rangle_K.   
\]
This satisfies $\ks(\tS_i) = [z_i]$. 
Set $B_j := f_j(q) \tS_j$, $j=1,\dots,m$. 
Then $\ks(B_j) = f_j(q) [z_j] = [w_j]$, 
$j=1,\dots,m$ satisfy the linear relations
\[
\sum_{j=1}^m \lan \varphi, b_j\ran [w_j] = 0 
\] 
for all $\varphi \in M$. 
Consequently,  
\begin{itemize} 
\item $B_j = f_j(q) \tS_j$; 
\item $f_j(q)=1$ if $b_j$ is a vertex
of the fan polytope (Proposition \ref{prop:vertex-vanishing}); 
\item $B_j$, $j=1,\dots,m$ satisfy the linear relations 
(by the injectivity of $\ks$). 
\end{itemize}  
By the characterization of the Batyrev elements 
(see Theorem \ref{thm:GI2}), 
we know that $B_j = \tD_j$, i.e.\ 
$f_j(q) = \exp(g_0^{(j)}(y))$. 
This shows the conjecture of Chan-Lau-Leung-Tseng: 
\begin{theorem} 
\label{thm:CLLT} 
Assume that the degeneration formula (Conjecture \ref{conj:degeneration}) 
holds for $(X,L)$ equipped with the $\C^\times$-actions $\rho_j$, 
$j=1,\dots,m$ in \S \ref{subsec:Seidel-toric}.  
Then Conjecture \ref{conj:CLLT} holds.  
\end{theorem} 

\begin{remark} 
Via the natural map $\bigoplus_{j=1}^m K z_j \to R$, 
the map $\ks$ induces the so-called Kodaira-Spencer 
map (cf.\ the discussion at the end of 
\S \ref{subsubsec:conjecture}):   
\[
\KS \colon 
H^2(X) \otimes K \to \Jac(W) 
\]
where the Jacobi algebra $\Jac(W)$ is defined to be 
\[
\Jac(W) := R/R \langle H^1(L), dW \rangle.  
\]
Then we have $\KS(\tS_i) = [z_i]$ and $\KS(\tD_i) = [w_i]$. 
In other words, the Seidel elements are the inverses of 
$[z_i]$ and the Batyrev elements are the inverses of $[w_i]$. 
\end{remark}

\subsection{Consistency check: computing equivariant Seidel elements} 
Here we give a consistency check concerning 
Chan-Lau-Leung-Tseng Conjecture \ref{conj:CLLT} 
and our degeneration formula 
\eqref{eq:degenerationformula-toric}.  
We calculate the lifts $\hS_j$ of Seidel elements 
assuming Conjecture \ref{conj:CLLT} and \eqref{eq:degenerationformula-toric} 
and see that the result is compatible with our previous 
calculation \cite{GoIr-Se12}. 
The lifts $\hS_j$ here should be viewed as the 
$\T$-equivariant Seidel elements since 
$H_\T^2(X) \cong H^2(X,L)$. 


\begin{lemma} 
\label{lem:w-Batyrev} 
Suppose that Conjecture \ref{conj:CLLT} holds. 
Then $w_i = f_i(q) z_i$, $i=1,\dots,m$ 
satisfy the multiplicative relation 
\[
\prod_{j=1}^m w_j^{\lan D_j, d\ran } = y^d \qquad 
\text{for all} \  
d \in H_2(X;\Z). 
\]
In other words, $y_a = \prod_{j=1}^m w_j^{m_{aj}}$, $a=1,\dots, r$.  
\end{lemma} 
\begin{proof} 
Recall that the Seidel and the Batyrev elements 
satisfy the multiplicative relations 
with respect to the quantum product (\S \ref{subsec:Seidel-toric}): 
\begin{align*} 
\prod_{j=1}^m \tD_j^{\lan D_j, d\ran} = y^d, 
\quad 
\prod_{j=1}^m \tS_j^{\lan D_j, d\ran} = q^d. 
\end{align*} 
Hence we have 
\[
\prod_{j=1}^m f_j(q)^{\lan D_j, d\ran} = y^d/q^d. 
\]
Therefore 
\[
\prod_{j=1}^m w_j^{\lan D_j, d\ran} = 
\prod_{j=1}^m \left( f_j(q)^{\lan D_j, d\ran } 
z_j^{\lan D_j, d\ran }\right) 
= (y^d/q^d ) \cdot q^d = y^d. 
\]
\end{proof} 

\begin{theorem} 
\label{thm:equivSeidel}
Assume Conjecture \ref{conj:CLLT} and the degeneration 
formula \eqref{eq:degenerationformula-toric}.  
The lifts $\hS_j$ of the Seidel elements are given by 
\[
\hS_j = e^{- g_0^{(j)}(y)}  
\left(  
D_j - \sum_{i=1}^m D_i 
\sum_{\substack{c_1(X) \cdot d =0, 
D_i \cdot d <0, \\ 
D_k \cdot d \ge 0 \ \text{\rm for all } k \neq i.}}  
(-1)^{\lan D_i, d\ran} \lan D_j, d\ran  
\frac{(-\lan D_i, d\ran -1)!}{
\prod_{k\neq i} \lan D_k, d\ran ! } y^d 
\right) 
\]
under the mirror transformation. 
\end{theorem} 
\begin{proof} 
Note that $(d w_1,\dots,d w_m)^{\rm T}$ 
can be viewed as the Jacobi matrix between the two 
co-ordinate systems $(w_1,\dots,w_m)$ 
and $(\log z_1,\dots,\log z_m)$ on the 
open-closed moduli space. 
The degeneration formula \eqref{eq:degenerationformula-toric} implies that 
$(z_1^{-1}\hS_1,\dots,z_m^{-1}\hS_m)$ 
is the inverse Jacobi matrix, i.e.\ 
\[
 z_j ^{-1}\hS_j =\sum_{i=1}^m \parfrac{\log z_i}{w_j} D_i 
 = w_j^{-1} \sum_{i=1}^m \parfrac{\log z_i}{\log w_j} D_i 
\]
in $H^2(X,L)$. 
Assuming Conjecture \ref{conj:CLLT}, we have 
$\log z_i = \log w_i - g_0^{(i)}(y)$. 
Hence 
\begin{align*} 
\hS_j & = \frac{z_j}{w_j} 
\sum_{i=1}^m \left(\delta_{ij} - w_j \parfrac{g_0^{(i)}}{w_j} 
\right) D_i \\ 
& = \exp\left(-g_0^{(j)}(y)\right)  
\left(D_j  - \sum_{i=1}^m 
\sum_{a=1}^r m_{aj} y_a \parfrac{g_0^{(i)}}{y_a} D_i
\right). 
\end{align*} 
In the second line, we used Lemma \ref{lem:w-Batyrev}. 
The conclusion follows. 
\end{proof} 

Note that we did not use the lifts $\hS_j$ of the Seidel 
elements (but used only the original Seidel elements $\tS_j$) 
in the proof of Theorem \ref{thm:CLLT}. 

\begin{remark} 
This result is compatible with the calculation 
of $\tS_j$ in our previous paper \cite{GoIr-Se12}. 
Note however that the formula in 
\cite[Lemma 3.17]{GoIr-Se12} 
contains a mistake. 
It occurred from an erroneous cancellation 
between the factors $\lan D_j, d\ran$ in the numerator 
and $\lan D_j, d\ran!$ in the denominator. 
\end{remark} 

\begin{remark} 
It is not difficult to generalize the computation in \cite{GoIr-Se12} 
to the $\T$-equivariant setting and to check the above 
computation of $\hS_j$ without using Conjecture \ref{conj:CLLT} 
and the degeneration formula \eqref{eq:degenerationformula-toric}. 
Since Chan-Lau-Leung-Tseng's conjecture \ref{conj:CLLT} was 
proved by themselves \cite{CLLT12}, it follows that the degeneration 
formula \eqref{eq:degenerationformula-toric} holds true in toric case.  
\end{remark}


\begin{thebibliography}{10}

\bibitem{At-Co82}
  \textsc{M.~F. Atiyah.}
  \newblock Convexity and commuting {H}amiltonians.
\newblock { {Bull. Lond. Math. Soc.}}, 14 (1982), 1--15.

\bibitem{Au-To04}
\textsc{M. Audin.}
\newblock { Torus actions on symplectic manifolds}.
Second revised edition. 
\newblock {{P}rogress in {M}athematics}, 93. 
\newblock Birkh\"auser Verlag, Basel, 2004.

\bibitem{Au-Mi07}
  \textsc{D. Auroux.}
  \newblock Mirror symmetry and {$T$}-duality in the complement of an
  anticanonical divisor.
  \newblock { J. G\"okova Geom. Topol.} GGT, 1 (2007) 51--91.

\bibitem{Au-Sp09}
  \textsc{D. Auroux.}
  \newblock Special {L}agrangian fibrations, wall-crossing, and mirror symmetry.
\newblock In { Surveys in differential geometry. {V}ol. {XIII}. {G}eometry,
  analysis, and algebraic geometry: forty years of the {J}ournal of
  {D}ifferential {G}eometry}, 1--47, { Surv. Differ.
  Geom.}, 13, Int. Press, Somerville, MA, 2009.

\bibitem{Ba-Qu93}
  \textsc{V.~V. Batyrev.}
  \newblock Quantum cohomology rings of toric manifolds.
\newblock Journ{\'e}es de G{\'e}om{\'e}trie Alg{\'e}brique d'Orsay (Orsay,
  1992).
\newblock { Ast\'erisque}, no. 218. (1993) 9--34.


\bibitem{Chan-formula}
  \textsc{K. Chan.}
  \newblock A formula equating open and closed {G}romov-{W}itten invariants and
  its applications to mirror symmetry.
\newblock { Pacific J. Math.}, 254 (2011), no. 2, 275--293.

\bibitem{ChLa-open}
  \textsc{K. Chan and S-C. Lau.}
  \newblock Open {G}romov-{W}itten invariants and superpotentials for semi-{F}ano
  toric surfaces.
\newblock Int. Math. Res. Not. IMRN (2014), no. 14, 3759--3789.


\bibitem{CLLT11}
  \textsc{K. Chan, S-C. Lau, N.C. Leung, and H-H. Tseng.}
  \newblock Open {G}romov-{W}itten invariants and mirror maps for semi-{F}ano
  toric manifolds.
\newblock arXiv:1112.0388.

\bibitem{CLLT12}
  \textsc{K. Chan, S-C. Lau, N.C. Leung, and H-H. Tseng.}
  \newblock Open {G}romov-{W}itten invariants and {S}eidel representations for
  toric manifolds.
\newblock arXiv:1209.6119.

\bibitem{ChOh-Fl06}
  \textsc{C-H. Cho and Y-G. Oh.}
  \newblock Floer cohomology and disc instantons of {L}agrangian torus fibers in
  {F}ano toric manifolds.
\newblock { Asian J. Math.} 10 (2006), no. 4, 773--814.

\bibitem{CoKa-Mi99}
  \textsc{D.~A. Cox and S. Katz.}
  \newblock { Mirror symmetry and algebraic geometry}, volume~68 of {
  Mathematical Surveys and Monographs}.
\newblock American Mathematical Society, Providence, RI, 1999.

\bibitem{CoLi-To10}
  \textsc{D.~A. Cox, J. Little, and H. Schenck.}
  \newblock { Toric Varieties}.
\newblock American Mathematical Society, 2010.

\bibitem{FuOh-La09}
  \textsc{K. Fukaya, Y-G. Oh, H. Ohta, and K. Ono.}
  \newblock { Lagrangian intersection {F}loer theory: anomaly and obstruction.
  {P}art {I}, {P}art {II}.}, volume~46 of { AMS/IP Studies in Advanced
  Mathematics}.
\newblock American Mathematical Society, Providence, RI, 2009.

\bibitem{FuOh-La10A}
  \textsc{K. Fukaya, Y-G. Oh, H. Ohta, and K. Ono.}
  \newblock Lagrangian {F}loer theory and mirror symmetry on compact toric
  manifolds.
\newblock 2010.
\newblock arXiv:1009.1648.

\bibitem{FuOh-La10}
  \textsc{K. Fukaya, Y-G. Oh, H. Ohta, and K. Ono.}
  \newblock Lagrangian {F}loer theory on compact toric manifolds. {I}.
\newblock Duke Math. J. 151 (2010), no. 1, 23--174.

\bibitem{FOOO-technical12}
  \textsc{K. Fukaya, Y-G. Oh, H. Ohta, and K. Ono.}
  \newblock Technical details on {K}uranishi structure and virtual fundamental
  chain.
\newblock 2012.
\newblock arXiv:1209.4410.

\bibitem{FOOO-toricdeg}
  \textsc{K. Fukaya, Y.-G. Oh, H. Ohta, and K. Ono.}
  \newblock Toric degeneration and non-displaceable {L}agrangian tori in
  ${S}^2\times {S}^2$.
\newblock arXiv:1002.1660.

\bibitem{Fukaya-Ono-Top99}
  \textsc{K. Fukaya and K. Ono.}
  \newblock Arnold conjecture and {G}romov-{W}itten invariant.
\newblock Topology 38 (1999), no. 5, 933--1048.

\bibitem{Gi-A-98}
  \textsc{A. Givental.}
  \newblock A mirror theorem for toric complete intersections.
\newblock In { Topological field theory, primitive forms and related topics
  ({K}yoto, 1996)}, volume 160 of { Progr. Math.}, pages 141--175.
  Birkh\"auser Boston, Boston, MA, 1998.

\bibitem{Go-Qu06}
  \textsc{E. Gonz\'alez.}
  \newblock Quantum cohomology and {$S^1$}-actions with isolated fixed points.
\newblock Trans. Amer. Math. Soc. 358 (2006), no. 7, 2927--2948 
(electronic). 

\bibitem{GoIr-Se12}
  \textsc{E. Gonz\'alez and H. Iritani.}
  \newblock Seidel elements and mirror transformations.
\newblock Selecta Math. (N.S.) 18 (2012), no. 3, 557--590.

\bibitem{Gu-Mo94}
  \textsc{V.~Guillemin.}
  \newblock { Moment Maps and Combinatorial Invariants of {H}amiltonian
  {$T^n$}-Spaces}, volume 122 of { Progress in Mathematics}.
\newblock Birkh\"{a}user, Boston, 1994.

\bibitem{KaLi-06}
  \textsc{S. Katz and C-C.~M. Liu.}
  \newblock Enumerative geometry of stable maps with {L}agrangian boundary
  conditions and multiple covers of the disc.
\newblock In {The interaction of finite-type and {G}romov-{W}itten
  invariants ({BIRS} 2003)}, volume~8 of { Geom. Topol. Monogr.}, pages
  1--47. Geom. Topol. Publ., Coventry, 2006.

\bibitem{LaMc-To99}
  \textsc{F. Lalonde, D. McDuff, and L. Polterovich.}
  \newblock Topological rigidity of {H}amiltonian loops and quantum homology.
\newblock Invent. Math. 135 (1999), no. 2, 69--385.

\bibitem{Mc-Qu00}
  \textsc{D. McDuff.}
  \newblock Quantum homology of fibrations over {$S^2$}.
\newblock { Internat. J. Math.}, 11 (2000), no. 5, 665--721.

\bibitem{McTo-To06}
  \textsc{D. McDuff and S. Tolman.}
  \newblock Topological properties of {H}amiltonian circle actions.
\newblock IMRP Int. Math. Res. Pap. (2006), 72826, 1--77.

\bibitem{Se-pi97}
  \textsc{P. Seidel.}
  \newblock {$\pi\sb 1$} of symplectic automorphism groups and invertibles in
  quantum homology rings.
\newblock Geom. Funct. Anal. 7 (1997), no. 6, 1046--1095.

\end{thebibliography}

\def\cprime{$'$}

\end{document}